\documentclass{amsart}

\usepackage{amsmath}
\usepackage{amssymb}
\usepackage[all]{xy}
\usepackage{mathrsfs}
\usepackage{upgreek}
\usepackage{stmaryrd}
\usepackage[colorlinks = true]{hyperref}
\usepackage{rotating}
\usepackage{tikz}
\usetikzlibrary{cd}

\usepackage{geometry}
\newcommand{\Av}{\mathrm{Av}}
\newcommand{\bY}{\mathbf{Y}}
\newcommand{\Norm}{\mathrm{N}}
\newcommand{\unip}{{\mathrm{u}}}

\newcommand{\Ga}{\mathbb{G}_{\mathrm{a}}}

\newcommand{\av}{\mathrm{av}}
\newcommand{\AS}{\mathrm{AS}}
\newcommand{\cX}{\mathcal{X}}

\newcommand{\cQ}{\mathcal{Q}}

\newcommand{\bk}{\Bbbk}
\newcommand{\Z}{\mathbb{Z}}

\newcommand{\F}{\mathbb{F}}

\newcommand{\sC}{\mathsf{C}}
\newcommand{\irr}{\mathsf{L}}

\newcommand{\Fr}{\mathrm{Fr}}

\newcommand{\bX}{\mathbf{X}}
\newcommand{\ext}{{\mathrm{ext}}}
\newcommand{\aff}{{\mathrm{aff}}}

\newcommand{\cF}{\mathcal{F}}

\newcommand{\cG}{\mathcal{G}}
\newcommand{\cL}{\mathcal{L}}

\newcommand{\IC}{\mathcal{IC}}

\newcommand{\cI}{\mathcal{I}}

\newcommand{\Gr}{\mathrm{Gr}}
\newcommand{\Fl}{\mathrm{Fl}}

\newcommand{\Perv}{\mathsf{Perv}}

\newcommand{\Rep}{\mathsf{Rep}}

\newcommand{\Db}{D^{\mathrm{b}}}

\DeclareMathOperator{\Hom}{Hom}
\DeclareMathOperator{\Ext}{Ext}
\DeclareMathOperator{\End}{End}

\DeclareMathOperator{\For}{For}

\newcommand{\simto}{\xrightarrow{\sim}}

\newcommand{\fR}{\mathfrak{R}}
\newcommand{\fRs}{\mathfrak{R}_{\mathrm{s}}}
\newcommand{\minR}{w^{\mathrm{R}}}
\newcommand{\minL}{w^{\mathrm{L}}}

\newcommand{\res}{\mathrm{res}}

\newcommand{\DFl}{\mathscr{D}}
\newcommand{\NFl}{\mathscr{N}}
\newcommand{\LFl}{\mathscr{L}}

\newcommand{\DGr}{\Delta}
\newcommand{\NGr}{\nabla}
\newcommand{\LGr}{\mathsf{L}}

\newcommand{\Satake}{\mathsf{Sat}}

\DeclareFontFamily{U}{mathx}{\hyphenchar\font45}
\DeclareFontShape{U}{mathx}{m}{n}{
      <5> <6> <7> <8> <9> <10>
      <10.95> <12> <14.4> <17.28> <20.74> <24.88>
      mathx10
      }{}
\DeclareSymbolFont{mathx}{U}{mathx}{m}{n}
\DeclareFontSubstitution{U}{mathx}{m}{n}
\DeclareMathAccent{\widecheck}{0}{mathx}{"71}
\newcommand{\bG}{\widecheck{\mathbf{G}}}
\newcommand{\bB}{\widecheck{\mathbf{B}}}
\newcommand{\bT}{\widecheck{\mathbf{T}}}

\makeatletter
\def\lotimes{\@ifnextchar_{\@lotimessub}{\@lotimesnosub}}
\def\@lotimessub_#1{\mathchoice{\mathbin{\mathop{\otimes}^L}_{#1}}%
  {\otimes^L_{#1}}{\otimes^L_{#1}}{\otimes^L_{#1}}}
\def\@lotimesnosub{\mathbin{\mathop{\otimes}^L}}
\makeatother

\newcommand{\wttimes}{\mathbin{\widetilde{\times}}}

\numberwithin{equation}{section}
\newtheorem{thm}{Theorem}[section]
\newtheorem{lem}[thm]{Lemma}
\newtheorem{prop}[thm]{Proposition}
\newtheorem{cor}[thm]{Corollary}
\newtheorem{conj}[thm]{Conjecture}
\theoremstyle{definition}

\theoremstyle{remark}
\newtheorem{rmk}[thm]{Remark}

\title{A geometric Steinberg formula}

\author{Pramod N. Achar}
\address{Department of Mathematics\\
  Louisiana State University\\
  Baton Rouge, LA 70803\\
  U.S.A.}
\email{pramod@math.lsu.edu}
\author{Simon Riche}
\address{Universit\'e Clermont Auvergne\\
  CNRS, LMBP\\
  F-63000 Clermont-Ferrand\\
  France}
\email{simon.riche@uca.fr}
 
\thanks{P.A. was supported by NSF Grant No.~DMS-1802241.  This project has received
funding from the European Research Council (ERC) under the European Union's Horizon 2020
research and innovation programme (grant agreements No~677147 and~101002592).
}

\dedicatory{Dedicated to the memory of Jim Humphreys}

\begin{document}

\begin{abstract}
We prove an isomorphism for simple perverse sheaves on the affine Grassmannian of a connected reductive algebraic group that is a geometric counterpart (in light of the Finkelberg-Mirkovi{\'c} conjecture) of the Steinberg tensor product formula for simple representations of reductive groups over fields of positive characteristic.
\end{abstract}

\maketitle

%%%%%%%%%%%%%%%%%%%%%%%%%%%%%%%%%%%%%%%%%%%%%%%%%%%%%%%%
\section{Introduction}
%%%%%%%%%%%%%%%%%%%%%%%%%%%%%%%%%%%%%%%%%%%%%%%%%%%%%%%%

%------------------------------------------------------
\subsection{Overview}
%------------------------------------------------------

The main result of the present paper is a formula expressing Iwahori-equivariant simple perverse sheaves on the affine Grassmannian of a connected reductive algebraic group in terms of convolution of simple perverse sheaves associated with ``restricted" elements of the affine Weyl group and simple perverse sheaves in the Satake category. In view of the Finkelberg--Mirkovi{\'c} conjecture, this can be viewed as a geometric counterpart of the Steinberg tensor product theorem for simple representations of reductive groups. One of our motivations for studying this question is that it allows us (using ideas from~\cite{abbgm}) to define and study a conjectural geometric model for blocks of representations of the Frobenius kernel of this reductive group; see~\cite{ar-pt2}.

%------------------------------------------------------
\subsection{The Finkelberg--Mirkovi{\'c} conjecture}
\label{ss:intro-FM}
%------------------------------------------------------

Before stating this result,
let us recall the Finkelberg--Mirkovi{\'c} conjecture.

Consider a connected reductive algebraic group $G$ over an algebraically closed field $\mathbb{F}$ of characteristic $p \neq 0$, with a choice of Borel subgroup $B \subset G$ and maximal torus $T \subset B$, and set $\bY=X_*(T)$. Let $\cL G$ be the loop group of $G$, let $\cL^+G$ be its arc group, and consider the affine Grassmannian $\Gr = \cL G / \cL^+G$. Next, let $\bk$ be either a finite field of characteristic $\ell \neq p$, or an algebraic closure of such a field. Then we can consider the category $\Perv_{\cL^+G}(\Gr,\bk)$ of $\cL^+G$-equivariant (\'etale) $\bk$-perverse sheaves of $\Gr$, which admits a natural structure of monoidal category with monoidal product $\star^{\cL^+G}$. Recall that the \emph{geometric Satake equivalence}~\cite{mv} provides an equivalence of monoidal categories
\[
 \mathsf{Sat} : (\Perv_{\cL^+G}(\Gr,\bk), \star^{\cL^+G}) \simto (\Rep(G^\vee_\bk),\otimes)
\]
where $G^\vee_\bk$ is a split connected reductive algebraic group over $\bk$, with a canonical maximal torus $T^\vee_\bk$ whose lattice of characters is $\bY$ and such that the root datum of $(G^\vee_\bk,T^{\vee}_\bk)$ is dual to that of $(G,T)$, and $\Rep(G^\vee_\bk)$ is its category of finite-dimensional algebraic representations. We will also denote by $B^{\vee}_\bk \subset G^\vee_\bk$ the Borel subgroup whose roots are the negative coroots of $(G,T)$ (with respect to our choice of $B$, considered as a negative Borel subgroup in $G$). We have a canonical autoequivalence
\[
\mathrm{sw} : \Perv_{\cL^+G}(\Gr,\bk) \simto \Perv_{\cL^+G}(\Gr,\bk)
\]
induced by the automorphism of $\cL G$ given by $g \mapsto g^{-1}$. 

Let us denote by $W$ the Weyl group of $(G,T)$ and by $\fR^\vee \subset \bY$ the coroot system of $(G,T)$, and consider the affine Weyl group $W_\aff := W \ltimes \Z \fR^\vee$ and the ``extended" version $W_\ext := W \ltimes \bY$. The group $W_\aff$ is known to admit a canonical generating subset $S_\aff$ (depending on the choice of $B$) such that $(W_\aff,S_\aff)$ is a Coxeter system, and $W_\ext$ is a semidirect product of $W_\aff$ by an abelian group $\Omega$ acting by Coxeter group automorphisms, and is naturally endowed with a length function. Let $I_\unip$ be the preimage of the unipotent radical of $B$ under the canonical morphism $\cL^+G \to G$; then 
%there exists a natural structure of quasi-Coxeter group on $W_\ext$ such that $W$ is a parabolic subgroup, and such that 
the $I_\unip$-orbits on $\Gr$
%the connected component $\Gr^\circ$ of $\Gr$ containing the base point 
are in a canonical bijection with the subset $W_\ext^S \subset W_\ext$ of elements $w$ which have minimal length in the coset $wW$.
Consider also a connected reductive algebraic group $\bG$ over $\bk$ whose Frobenius twist $\bG^{(1)}$ is $G^\vee_\bk$, and denote by $\bT \subset \bB \subset \bG$ the maximal torus and Borel subgroup such that $\bT^{(1)}=T^\vee_\bk$ and $\bB^{(1)}=B^\vee_\bk$. The Frobenius morphism of $\bG$ (or of any of its subgroups) will simply be denoted $\Fr$.

We identify the character lattice of $\bT$ with $\bY$, in such a way that the pullback under the Frobenius morphism $\Fr : \bT \to T^{\vee}_\bk$ is given by $\lambda \mapsto \ell\lambda$. Let $\bY_+ \subset \bY$ be the set of dominant weights for $\bG$ (or dominant coweights for $G$) with respect to the choice of positive roots that makes $\bB$ the negative Borel subgroup.  For $\lambda \in \bY_+$ we denote by $\irr(\lambda)$ the simple $\bG$-module of highest weight $\lambda$.

The group $W_\aff$ is the affine Weyl group of $\bG$ in the sense of~\cite{jantzen}. We will denote the ``dot action'' of $W_\aff$ and $W_\ext$ on $\bY$ by $\cdot_\ell$. If $\ell \geq h$ where $h$ is the Coxeter number of $\bG$, we can consider the extended principal block $\Rep_{[0]}(\bG)$ in the category $\Rep(\bG)$ of finite-dimensional algebraic $\bG$-modules, namely the Serre subcategory generated by the simple $\bG$-modules of the form $\irr(w^{-1} \cdot_\ell 0)$ with $w \in W_\ext^S$.
%whose highest weight belongs to the intersection of $W_\aff \cdot_\ell 0$ with the dominant cone, i.e.~is of the form $w \cdot_\ell 0$ where $w^{-1} \in W_\aff^S$.

The following statement is known as the Finkelberg--Mirkovi{\'c} conjecture. Here we consider the category $\Perv_{I_\unip}(\Gr,\bk)$ of $I_\unip$-equivariant $\bk$-perverse sheaves on $\Gr$, with the natural convolution action (again denoted $\star^{\cL^+G}$) of the category $\Perv_{\cL^+G}(\Gr,\bk)$.
% of $\cL^+G$-equivariant $\bk$-perverse sheaves on $\Gr^\circ$.  
 The simple $I_\unip$-equivariant perverse sheaf supported on the closure of the $I_\unip$-orbit labeled by $w \in W_\ext^S$ is denoted by $\LGr_w$.

\begin{conj}[Finkelberg--Mirkovi{\'c} conjecture, \cite{fm}]
\label{conj:FM}
Assume that $\ell \geq h$, and that $\bY/\Z\fR^\vee$ has no $\ell$-torsion.
There exists an equivalence of categories
\[
 \mathsf{FM} : \Perv_{I_\unip}(\Gr,\bk) \simto \Rep_{[0]}(\bG)
\]
which identifies the natural highest weight structures on both sides, and satisfies
\[
\mathsf{FM}(\LGr_w) \cong \irr(w^{-1} \cdot_\ell 0)
\qquad \text{for any $w \in W_\ext^S$.}
\]
%sends for any $w \in W_\aff^S$ the intersection cohomology complex associated with the orbit labeled by $w$ to the simple $\bG$-module of highest weight $w^{-1} \cdot_\ell 0$. 
Moreover, for $\cF$ in $\Perv_{I_\unip}(\Gr,\bk)$ and $\cG$ in $\Perv_{\cL^+G}(\Gr,\bk)$ there exists a bifunctorial isomorphism
\[
 \mathsf{FM}(\cF \star^{\cL^+G} \cG) \cong \mathsf{FM}(\cF) \otimes \Fr^* \bigl( \mathsf{Sat}(\mathrm{sw}^* \cG) \bigr).
\]
\end{conj}

\begin{rmk}
\phantomsection
\label{rmk:intro-op}
\begin{enumerate}
\item
 The combinatorics involved in Conjecture~\ref{conj:FM} takes a more natural form if we work with the ``opposite'' affine Grassmannian $\Gr^{\mathrm{op}}=\cL^+G \backslash \cL G$ (with its action of $\cL^+G$ and $I_\unip$ induced by right multiplication on $\cL G$). It is however much more common to work with $\Gr$ rather than $\Gr^{\mathrm{op}}$, and for this reason we will work with the conjecture as formulated in Conjecture~\ref{conj:FM}.
\item
\label{it:conj-FM-ext}
If $\ell \geq h$, the group $\bY/\Z\fR^\vee$ can have $\ell$-torsion only if $G$ has a component on type $A_\ell$. This can create troubles with Conjecture~\ref{conj:FM}; e.g., the extended principal block of $\mathrm{SL}_\ell$ in characteristic $\ell$ has its simple objects in a natural bijection with $W_\ext^S \cap W_\aff$, which does not match the combinatorics of the category $\Perv_{I_\unip}(\Gr,\bk)$ for $G=\mathrm{PGL}_\ell$.
%It is often convenient in Conjecture~\ref{conj:FM} to replace $\Gr^\circ$ by $\Gr$ and $\Rep_0(\bG)$ by the ``extended principal block'' defined in terms of the extended affine Weyl group $W_\ext=W \ltimes \bY$ rather than $W_\aff$. The two versions of the conjecture are equivalence if $p>h$, but one must be careful with type $A_p$ in characteristic $p$. (For instance, the extended principal block coincides with the principal block for $\mathrm{SL}_p$ in characteristic $p$, while the similar statement is not true for perverse sheaves on $\Gr$.)
%It is tempting to state Conjecture~\ref{conj:FM} for the \emph{extended} principal block of $\bG$ (defined using the extended affine Weyl group $W_\ext=W \ltimes \bY$ instead of $W_\aff$) and the category $\Perv_{I_\unip}(\Gr,\bk)$. However, it turns out that this version is false when $G=\mathrm{PGL}(2)$ and $\ell=2$. Indeed, this version of the conjecture would imply that the natural morphism
%\[
%\Ext^1_{\bG^{(1)}}(V,V') \to \Ext^1_{\bG}(\Fr^*(V),\Fr^*(V'))
%\]
%is an isomorphism for all algebraic $\bG^{(1)}$-modules $V,V'$ with $\dim(V)<\infty$, but this is false in this case when $V=\bk$ and $V'=\scO(\bG^{(1)})$ by~\cite[Remark in~\S II.12.2]{jantzen}. We expect this ``extended" version of the Finkelberg--Mirkovi{\'c} conjecture to hold under slightly stronger assumptions on $\ell$.
\item
A proof of Conjecture~\ref{conj:FM} seems within reach (maybe under stronger assumptions), but is not available as of now.
\end{enumerate}
\end{rmk}

%------------------------------------------------------
\subsection{The geometric Steinberg formula}
%------------------------------------------------------

From now on we assume for simplicity that the center of $G$ is a torus. (Most questions we are interested in can be reduced to this case.)
%e will denote by $\bY_+ \subset \bY$ the subset of coweights of $G$ which are dominant. Under the identification $\bY=X^*(\bT^{(1)})$, resp.~$\bY=X^*(\bT)$, $\bY_+$ corresponds to the dominant weights for the choice of positive roots of $\bG^{(1)}$, resp.~$\bG$, such that $\bB^{(1)}$, resp.~$\bB$, is the negative Borel subgroup.  For $\lambda \in \bY_+$ we denote by $\irr(\lambda)$, resp.~$\irr^{(1)}(\mu)$, the simple $\bG$-module, resp.~$\bG^{(1)}$-module, of highest weight $\lambda$ (considered as a character of $\bT$, resp.~$\bT^{(1)}$).
The main results of the paper are statements which correspond under Conjecture~\ref{conj:FM} (and some ``singular" analogues) to the following two classical results in representation theory.  (Here, $\irr^{(1)}(\lambda)$ denotes the simple $\bG^{(1)}$-module of highest weight $\lambda$, and $\bG_1$ denotes the Frobenius kernel of $\bG$.)
%where $\bG_1$ is the Frobenius kernel of $\bG$ and for $\lambda \in \bY_+$ we denote by $\irr(\lambda)$, resp.~$\irr^{(1)}(\mu)$, the simple $\bG$-module, resp.~$\bG^{(1)}$-module, of highest weight $\lambda$ (considered as a character of $\bT$, resp.~$\bT^{(1)}$).
\begin{enumerate}
\item
\label{it:intro-Steinberg}
(Steinberg's tensor product formula,~\cite[Proposition~II.3.16]{jantzen})
For any $\lambda \in \bY_+$ restricted and any $\mu \in \bY_+$ we have
\[
\irr(\lambda + \ell\mu) \cong \irr(\lambda) \otimes \Fr^*(\irr^{(1)}(\mu)).
\]
\item
\label{it:intro-Steinberg-ff}
(\cite[Propositions~II.3.10 and II.3.15]{jantzen}) For $\lambda \in \bY_+$ restricted, 
$\End_{\bG_1}(\irr(\lambda))=\bk$.
\end{enumerate}
More precisely, instead of~\eqref{it:intro-Steinberg-ff} we will prove an analogue of the following equivalent statement:
\begin{enumerate}
\setcounter{enumi}{2}
\item
\label{it:intro-Steinberg-ff-2}
For any $\lambda \in \bY_+$ restricted, the functor
\[
\Rep(G^\vee_\bk) \to \Rep(\bG)
\]
defined by $V \mapsto \irr(\lambda) \otimes \Fr^*(-)$ is fully faithful.
\end{enumerate}

The geometric counterpart of $\irr^{(1)}(\mu)$ (for $\mu \in \bY_+$) is provided by the geometric Satake equivalence: it is a classical fact that the simple objects in $\Perv_{\cL^+G}(\Gr,\bk)$ are in canonical bijection with $\bY_+$, and that if we denote by $\IC^\mu$ the simple object attached to $\mu \in \bY_+$, then we have $\Satake(\IC^\mu) = \irr^{(1)}(\mu)$ for any $\mu \in \bY_+$.
%Now, let $W_\ext := W \ltimes \bY$ be the extended affine Weyl group (as in Remark~\ref{rmk:intro-op}\eqref{it:conj-FM-ext}), which contains $W_\aff$ as a normal subgroup.
%; it is known that the length function $\ell$ of $W_\aff$ (defined from the Coxeter system $(W_\aff, S_\aff)$) admits a canonical extension to $W_\ext$, which will be denoted similarly. 
Using the geometry of alcoves one can naturally define a subset $W_\ext^\res \subset W_\ext$ of ``restricted elements'' which provides a replacement for the restricted dominant weights for $G^\vee_\bk$; see~\S\ref{ss:restricted} for details. This subset is contained in 
%the subset 
$W_\ext^S$.
% of elements $w$ which are minimal in $wW$ (for the Bruhat order).
%, which 
%is itself stable under right multiplication by the elements $t_\lambda:=(e \ltimes \lambda)$ with $\lambda \in \bY$ antidominant.

We can now state the main result of the paper, which provides a geometric counterpart to the properties~\eqref{it:intro-Steinberg}--\eqref{it:intro-Steinberg-ff-2} above.  In this statement, $w_\circ$ denotes the longest element of $W$, and $t_\lambda := e \ltimes \lambda$ denotes the element of $W_\ext$ corresponding to a weight $\lambda \in \bY$.

\begin{thm}
\label{thm:main-intro}
For any $w \in W_\ext^\res$, the functor
\[
\LGr_w \star^{\cL^+ G} (-) : \Perv_{\cL^+G}(\Gr,\bk) \to \Perv_{I_\unip}(\Gr,\bk)
\]
is fully faithful, and satisfies
$\LGr_w \star^{\cL^+ G} \IC^\mu \cong \LGr_{wt_{w_\circ(\mu)}}$ for any $\mu \in \bY_+$.
\end{thm}

(We remark that $W_\ext^S$ is stable under multiplication on the right by $t_\lambda$ for $\lambda$ antidominant, so that $wt_{w_\circ(\mu)}$ is a valid label of an $I_\unip$-orbit.)

Our proof of Theorem~\ref{thm:main-intro} follows arguments found in~\cite{abbgm}, where the authors prove the isomorphism $\LGr_w \star^{\cL^+ G} \IC^\mu \cong \LGr_{wt_{w_\circ(\mu)}}$ when $\bk$ is an algebraic closure of $\mathbb{Q}_\ell$. As presented there, the proof uses some special features of the characteristic-$0$ setting (e.g.~the decomposition theorem); however, a closer analysis of their arguments reveals that they prove the full faithfulness statement in Theorem~\ref{thm:main-intro} in the case of positive-characteristic coefficients too. (Note that when $\bk$ has characteristic $0$ the Satake category is semisimple, so that this full faithfulness statement is an immediate consequence of the isomorphism of simple perverse sheaves.) It is then not difficult to deduce the isomorphism for simple objects.

\begin{rmk}
\begin{enumerate}
\item
Let us emphasize that there is no assumption on $\ell$ in Theorem~\ref{thm:main-intro}. Such an assumption is needed only to (conjecturally) relate this statement to Representation Theory.
\item
In the body of the paper we will also prove a ``Whittaker" variant of Theorem~\ref{thm:main-intro}, for any choice of a subset $A \subset S_\aff$ generating a finite subgroup. (The case stated above corresponds to $A=\varnothing$.) This is motivated by a ``singular" variant of Conjecture~\ref{conj:FM}, which postulates the existence of a similar equivalence relating a singular block of $\Rep(\bG)$ (with ``singularity'' determined by $A$) with a category of perverse sheaves on $\Gr$ satisfying a Whittaker condition relative to a group attached to $A$.
\item
In this paper we work with perverse sheaves for the \'etale topology because we want to cover also the ``Whittaker'' categories, which have no counterpart at this point in the ``classical" setting of perverse sheaves for the analytic topology. However, in case $A=\varnothing$ our category is just the category of $I_\unip$-equivariant perverse sheaves on $\Gr$, which also makes sense in the classical setting; in this special case our proof of Theorem~\ref{thm:main-intro} applies in both settings. 
\end{enumerate}
\end{rmk}

%------------------------------------------------------
\subsection{Contents}
%------------------------------------------------------

In Section~\ref{sec:combinatorics} we prove a number of preliminary results of a combinatorial nature regarding the extended affine Weyl group $W_\ext$.
In Section~\ref{sec:perverse} we define our categories of perverse sheaves $\Perv_{(I_\unip^A, \cX_A)}(\Gr,\bk)$ (and their analogues for sheaves on the affine flag variety $\Fl$) and the ``averaging'' functors relating them. All the results from these sections are known in some form, but we found it convenient to state them and give (sketches of) proofs.
Finally, in Section~\ref{sec:geometric-Steinberg} we prove Theorem~\ref{thm:main-intro}. 

%%%%%%%%%%%%%%%%%%%%%%%%%%%%%%%%%%%%%%%
\section{Combinatorics of the affine Weyl group}
\label{sec:combinatorics}
%%%%%%%%%%%%%%%%%%%%%%%%%%%%%%%%%%%%%%%

%-------------------------------------------------------------------
\subsection{The extended affine affine Weyl group}
\label{ss:Wext}
%-------------------------------------------------------------------

Let $\F$, $G$, $B$, $T$, $\bY$, $W$ be as in~\S\ref{ss:intro-FM}.
We will denote by $\bX:=X^*(T)$ the character lattice of $T$, by $\fR \subset \bX$ the root system of $(G,T)$, and by $\fR^\vee \subset \bY$ the coroot system; the natural bijection from $\fR$ to $\fR^\vee$ will be denoted $\alpha \mapsto \alpha^\vee$ as usual.
We will denote by $\fR_+ \subset \fR$ the system of positive roots consisting of the $T$-weights in $\mathrm{Lie}(G)/\mathrm{Lie}(B)$, and by $\fRs$ the associated basis of $\fR$. 
The corresponding sets of dominant coweights and strictly dominant coweights will be denoted $\bY_+$ and $\bY_{++}$ respectively. 
If we denote by $S \subset W$ the subset consisting of the reflections $s_{\alpha^\vee}$ for $\alpha \in \fRs$, then it is well known that $(W,S)$ is a Coxeter system. The longest element in this group will be denoted $w_\circ$. We will assume that $\bX/\Z\fR$ has no torsion, or in other words
that the scheme-theoretic center of $G$ is connected.
This condition ensures that there exists $\varsigma \in \bY$ such that $\langle \alpha,\varsigma \rangle = 1$ for all $\alpha \in \fR_{\mathrm{s}}$; we fix such an element once and for all.

The affine Weyl group associated with $G$
is the semidirect product
\[
W_\aff := W \ltimes \Z\fR^\vee,
\]
where $\Z\fR^\vee \subset \bY$ is the lattice generated by $\fR^\vee$. For $\lambda \in \Z\fR^\vee$, we will write $t_\lambda$ for the corresponding element of $W_\aff$. It is a standard fact that if we denote by $S_\aff \subset W_\aff$ the subset consisting of $S$ together with the elements $t_{\beta^\vee} s_{\beta^\vee}$ where $\beta^\vee \in \fR^\vee$ is a maximal short coroot, then the pair $(W_\aff,S_\aff)$ is a Coxeter system. Moreover, classical results of Iwahori--Matsumoto~\cite{im} show that the associated length function on $W_\aff$ can be described by the following formula for $w \in W$ and $\lambda \in \Z\fR^\vee$:
\begin{equation}
\label{eqn:formula-length}
\ell(w t_\lambda) = \sum_{\substack{\alpha \in \fR_+ \\ w(\alpha) \in \fR_+}} |\langle \lambda,\alpha \rangle| + \sum_{\substack{\alpha \in \fR_+ \\ w(\alpha) \in -\fR_+}} |1+\langle \lambda,\alpha \rangle|.
\end{equation}

The formula on the right-hand side of~\eqref{eqn:formula-length} makes sense more generally for $\lambda \in \bY$, which allows us to extend the function $\ell$ to the larger group
\[
W_\ext := W \ltimes \bY.
\]
The subgroup $W_\aff \subset W_\ext$ is normal, and if we set
\[
\Omega:=\{w \in W_\ext \mid \ell(w)=0\}
\]
then $\Omega$ is a finitely generated abelian group acting on $W_\aff$ (via conjugation) by Coxeter group automorphisms, and multiplication induces a group isomorphism
\[
\Omega \ltimes W_\aff \simto W_\ext;
\]
moreover $\ell(\omega w)=\ell(w \omega)=\ell(w)$ for any $w \in W_\ext$ and $\omega \in \Omega$. We can also extend the Bruhat order $\leq$ on $W_\aff$ to $W_\ext$ by declaring that for $\omega,\omega' \in \Omega$ and $w,w' \in W_\aff$ we have $\omega w \leq \omega' w'$ iff $\omega=\omega'$ and $w \leq w'$. (The same rule will then also apply when switching the order of $\omega$ and $w$.) 

The following property holds for general Coxeter groups, and can be checked using the characterization of the Bruhat order in terms of reduced expressions and the exchange condition.

\begin{lem}
\label{lem:Bruhat-order-lengths-add}
Let $x,y,w \in W_\ext$, and assume that $\ell(xy)=\ell(x)+\ell(y)$ and $\ell(xw)=\ell(x)+\ell(w)$. Then
$y \leq w$ if and only if $xy \leq xw$.
\end{lem}

%-------------------------------------------------------------------
\subsection{Coset representatives}
\label{ss:coset}
%-------------------------------------------------------------------

If $A \subset S_\aff$ is a subset, we will denote by $W_A$ the subgroup of $W_\aff$ generated by $A$; if this subgroup is finite we will say that $A$ is \emph{finitary}, and we will denote by $w_A$ the longest element in $W_A$. In this case, the theory of Coxeter systems guarantees that for any $w \in W_\ext$ the cosets $W_A w$ and $wW_A$ each admit a unique minimal element (and a unique maximal element) with respect to the Bruhat order. If $w$ is minimal in $W_A w$, resp.~in $wW_A$, then for any $x \in W_A$ we have
$\ell(xw)=\ell(x)+\ell(w)$, resp.~$\ell(wx)=\ell(w)+\ell(x)$.
In fact, it is easily seen that
\begin{equation}
\label{eqn:min-length}
\text{$w$ is minimal in $W_A w$ iff $\ell(w_A w)=\ell(w_A)+\ell(w)$.}
\end{equation}

The following claim is well known (see e.g.~the discussion in~\cite[p.~86]{soergel}).

\begin{lem}
\label{lem:minimal-mult-Saff}
Let $w \in W_\ext$ be an element which is minimal in $wW_A$. If $s \in S_\aff$ and $sw$ is not minimal in $swW_A$, then $sw=wr$ for some $r \in A$; in particular, if $s \in S_\aff$ satisfies $sw<w$, then $sw$ is minimal in $swW_A$.
\end{lem}

Below we will consider the restriction of the Bruhat order to the subset of elements $w$ in $W_\ext$ which are minimal in $wW_A$ (resp.~in $W_Aw$). If $y,w \in W_\ext$ are minimal in their respective cosets $yW_A$ and $wW_A$, and if $y',w'$ are the maximal elements in these cosets, it is a standard fact (see~\cite[Lemma~2.2]{douglass}) that the following conditions are equivalent:
\begin{enumerate}
\item 
\label{eqn:order-representatives-1}
$y \leq w$;
\item 
\label{eqn:order-representatives-2}
$y' \leq w'$;
\item 
\label{eqn:order-representatives-3}
there exists $y'' \in yW_A$ and $w'' \in wW_A$ such that $y'' \leq w'$'.
\end{enumerate}
Of course, a similar property holds for cosets in $W_A \backslash W_\ext$.

\begin{rmk}
\label{rmk:order}
One can similarly consider minimal and maximal elements in double cosets of the form $W_A w W_{A'}$ where $A,A' \subset S_\aff$ are finitary subsets. The analogues of~\eqref{eqn:order-representatives-1}--\eqref{eqn:order-representatives-3} are also equivalent in this setting, as proved in~\cite[Lemma~2.2]{douglass}. 
%Below we will use in particular the fact that if $y,w \in W_\ext$ are minimal in their cosets $WwW$ and $WyW$, and if $y',w'$ are the maximal elements in $WwW$ and $WyW$ respectively, then $y \leq w$ iff $y' \leq w'$.
\end{rmk}

In particular, we will consider these notions in the case $A=S$, so that $W_A=W$. (In this case, we have already introduced the notation $w_\circ$ for the longest element in $W$, so that the notation $w_S$ will not be used.) The maximal and minimal elements in cosets
can be described explicitly in this case, as follows. First one notices that the quotients $W_\ext / W$ and $W \backslash W_\ext$ are in canonical bijection with $\bY$, so that every right coset is of the form $W t_\lambda$, and likewise for left cosets. For any $\lambda \in \bY$ the minimal element in $Wt_\lambda$, resp.~$t_\lambda W$, will be denoted $\minL_\lambda$, resp.~$\minR_\lambda$; we will also denote by $\mathsf{dom}(\lambda)$ the unique dominant $W$-translate of $\lambda$. By~\cite[Lemma~2.4]{mr} we have 
\begin{equation}
\label{eqn:formula-minL}
\minL_\lambda = v_\lambda t_\lambda = t_{\mathsf{dom}(\lambda)} v_\lambda,
\end{equation}
where $v_\lambda \in W$ is the element of minimal length such that $v_\lambda(\lambda)=\mathsf{dom}(\lambda)$. We moreover have
\[
\ell(\minL_\lambda)=\ell(t_\lambda)-\ell(v_\lambda)=\ell(t_{\mathsf{dom}(\lambda)})-\ell(v_\lambda).
\]
We clearly have 
\begin{equation}
\label{eqn:formula-minR}
\minR_\lambda=(\minL_{-\lambda})^{-1},
\end{equation}
and the maximal element in $Wt_\lambda$, resp.~in $t_\lambda W$, is $w_\circ \minL_\lambda$, resp.~$\minR_\lambda w_\circ$.

We will denote by $W_\ext^S \subset W_\ext$ the subset of elements $w$ which are minimal in their coset $wW$; we therefore have
$W_\ext^S = \{\minR_\lambda : \lambda \in \bY\}$.

In the following lemma we characterize the elements in $W_\ext^S$ which satisfy a certain minimality property with respect to left multiplication by elements of $W_A$. (This statement makes sense, and holds true with identical proof, for any choice of a Coxeter system and a finitary pair of subsets of the simple reflections.)

\begin{lem}
\label{lem:double-min}
Let $A \subset S_\aff$ be a finitary subset. For $w \in W_\ext$, the following conditions are equivalent:
\begin{enumerate}
\item 
\label{it:min-0}
$w \in W_\ext^S$ and $wv$ is minimal in $W_A wv$ for any $v \in W$;
\item 
\label{it:min-1}
$w \in W_\ext^S$ and $ww_\circ$ is minimal in $W_A w w_\circ$;
\item 
\label{it:min-2}
$w$ is minimal in $W_A w$ and $v w \in W_\ext^S$ for any $v \in W_A$;
\item 
\label{it:min-3}
$w$ is minimal in $W_A w$ and $w_A w \in W_\ext^S$;
\item 
\label{it:min-4}
$\ell(w_A w w_\circ) = \ell(w_A) + \ell(w) + \ell(w_\circ)$.
\end{enumerate}
\end{lem}

We will denote by ${}^A W^S_\ext \subset W_\ext^S$ the subset of elements which satisfy the conditions of Lemma~\ref{lem:double-min}.

\begin{proof}
Of course~\eqref{it:min-0} implies~\eqref{it:min-1}, and~\eqref{it:min-1} implies~\eqref{it:min-0} by (the right-coset analogue of) Lemma~\ref{lem:minimal-mult-Saff}. It is clear that~\eqref{it:min-1} implies~\eqref{it:min-4}. If~\eqref{it:min-4} holds, then by~\eqref{eqn:min-length} $w w_\circ$ is minimal in $W_A w w_\circ$. We also deduce that
\[
\ell(w w_\circ) \geq \ell(w_A w w_\circ) - \ell(w_A) = \ell(w) + \ell(w_\circ),
\]
hence $\ell(w w_\circ)=\ell(w) + \ell(w_\circ)$, which implies that $w$ belongs to $W_\ext^S$ by~\eqref{eqn:min-length}.
The equivalence with~\eqref{it:min-2}--\eqref{it:min-3} is obtained similarly, switching the roles of $A$ and $S$.
\end{proof}

As explained in Remark~\ref{rmk:order},
there is a general theory of minimal elements in double cosets in Coxeter groups. If $w \in {}^A W^S_\ext$ then $w$ is minimal in $W_A w W$; however not every element which is minimal in its double coset belongs to ${}^A W^S_\ext$.  Specifically, one can show that the minimal element of a double coset $W_A w W$ lies in ${}^A W^S_\ext$ if and only if the set $W_A w W \cap W^S_\ext$ has cardinality equal to that of $W_A$.  In the special case where $A = S$, the following lemma gives another description of this set.

\begin{lem}
\label{lem:min-LR}
Let $\lambda \in \bY$. We have $\minR_\lambda \in {}^S W_\ext^S$ iff
$\lambda \in \bY_{++}$. Moreover in this case we have $\minR_\lambda=t_\lambda w_\circ$.
\end{lem}

\begin{proof}
We will use the characterization of ${}^S W^S_\ext$ given by condition~\eqref{it:min-1} in Lem\-ma~\ref{lem:double-min}.
Using~\eqref{eqn:formula-minL} and~\eqref{eqn:formula-minR} we see that
\[
\minR_\lambda w_\circ = (v_{-\lambda})^{-1} t_{-\mathsf{dom}(-\lambda)} w_\circ = (v_{-\lambda})^{-1} w_\circ t_{-w_\circ\mathsf{dom}(-\lambda)} = (v_{-\lambda})^{-1} w_\circ t_{\mathsf{dom}(\lambda)}
\]
since $\mathsf{dom}(-\lambda)=-w_\circ\mathsf{dom}(\lambda)$.
Now $\minL_{\mathsf{dom}(\lambda)}=t_{\mathsf{dom}(\lambda)}$ by~\eqref{eqn:formula-minL}, so that $\minR_\lambda w_\circ$ is minimal in $W\minR_\lambda w_\circ$ iff $(v_{-\lambda})^{-1} w_\circ=e$, i.e.~iff $v_{-\lambda}=w_\circ$. This is clearly equivalent to the condition that $-\lambda \in -\bY_{++}$, i.e.~that $\lambda \in \bY_{++}$.
\end{proof}

%--------------------------------------------------------------
\subsection{Alcoves}
\label{ss:alcoves}
%--------------------------------------------------------------

Consider the vector space $V:=\bY \otimes_\Z \mathbb{R}$, and the action of $W_\ext$ given by
$(t_\lambda w) \cdot v = w(v) + \lambda$
for $w \in W$ and $\lambda \in \bY$, where $W$ acts on $V$ via its natural action on $\bY$. In $V$ we have the affine hyperplanes defined by
\[
H_{\beta,n} := \{v \in V \mid \langle \beta,v \rangle = n\}
\]
for $\beta \in \fR$ and $n \in \Z$, which are permuted by the action of $W_\ext$. The connected components of the complement of the union of these hyperplanes are called alcoves; if we set
\[
\mathfrak{A}_{\mathrm{fund}} := \{v \in V \mid \forall \beta \in \fR_+, \, 0 < \langle \beta,v \rangle < 1\},
\]
then $\mathfrak{A}_{\mathrm{fund}}$ is an alcove (called the \emph{fundamental alcove}), and moreover if we denote by $\mathscr{A}$ the set of alcoves then the assignment $w \mapsto w(\mathfrak{A}_{\mathrm{fund}})$ induces a bijection
\[
W_\ext / \Omega \simto \mathscr{A},
\]
where $\Omega$ is as in~\S\ref{ss:Wext}. If
\[
\mathcal{C} = \{v \in V \mid \forall \beta \in \fR_+, \, \langle \beta,v \rangle > 0\},
\]
then it is a standard fact that
\begin{equation}
\label{eqn:min-Wext-Afund}
W_\ext^S = \{w \in W_\ext \mid w^{-1}(\mathfrak{A}_{\mathrm{fund}}) \subset \mathcal{C}\}.
\end{equation}

%------------------------------------------------
\subsection{Restricted elements}
\label{ss:restricted}
%------------------------------------------------

For $\mu \in \bY$ we set
\[
\Pi_\mu := \{v \in V \mid \forall \alpha \in \fR_{\mathrm{s}}, \, \langle \alpha,\mu \rangle -1 < \langle \alpha, v \rangle < \langle \alpha,\mu \rangle \};
\]
our assumption on $\bX/\Z\fR$ ensures that each alcove is contained in a subset of this form. We define the subset of restricted elements in $W_\ext$ by setting
\[
W_\ext^\res := \{w \in W_\ext \mid w^{-1}(\mathfrak{A}_{\mathrm{fund}}) \subset \Pi_\varsigma\}.
\]
(Of course, this subset does not depend on the choice of $\varsigma$.)
The relation between $W_\ext^\res$ and $W_\ext^S$ is as follows: if $w \in W_\ext$, there exists $\mu \in \bY$ such that $w^{-1}(\mathfrak{A}_{\mathrm{fund}}) \subset \Pi_\mu$; then $w \in W_\ext^S$ if and only if $\mu \in \bY_{++}$. With this notation we have $t_{\varsigma-\mu} w^{-1}(\mathfrak{A}_{\mathrm{fund}}) \subset \Pi_\varsigma$, i.e.~$w t_{\mu-\varsigma}  \in W_\ext^\res$, and of course
\[
w = (w t_{\mu-\varsigma}) t_{\varsigma-\mu}.
\]
Here $\mu \in \bY_{++}$ iff $\varsigma-\mu \in -\bY_+$. In conclusion, we have shown that
\begin{equation}
\label{eqn:WS-Wres}
W_\ext^S = \{x t_{\lambda} : x \in W_\ext^\res, \, \lambda \in -\bY_+\}.
\end{equation}
(In case $G$ is semisimple, each element of $W_\ext^S$ can be written uniquely as a product $x t_{\lambda}$ with $x \in W_\ext^\res$ and $\lambda \in -\bY_+$, but in general this expression is not unique.) We will see in Lemma~\ref{lem:length-res-dom} below that lengths always add in such an expression. These considerations also show that for any fixed $w \in W_\ext$, there exists $\lambda \in \bY$ such that $w t_\lambda  \in W_\ext^\res$; in fact the elements that satisfy this property form a torsor for the lattice of elements in $\bY$ orthogonal to all roots.

\begin{lem}
\label{lem:res-elements}
Let $w \in W$, $\lambda \in \bY$. Then $w t_\lambda \in W_\ext^\res$ if and only if for all $\alpha \in \fR_{\mathrm{s}}$ we have
\[
\langle \alpha,\lambda \rangle = \begin{cases}
0 & \text{if $w(\alpha) \in \fR_+$;} \\
-1 & \text{if $w(\alpha) \in -\fR_+$.} 
\end{cases}
\]
In particular, if $w t_\lambda \in W_\ext^\res$ then $\lambda \in -\bY_+$.
\end{lem}

\begin{proof}
For $N \gg 0$ we have $\frac{1}{N} \varsigma \in \mathfrak{A}_{\mathrm{fund}}$; hence $w t_\lambda$ belongs to $W_\ext^\res$ if and only if $(t_{-\lambda} w^{-1}) \cdot (\frac{1}{N} \varsigma) \in \Pi_\varsigma$. Now we have
\[
\textstyle (t_{-\lambda} w^{-1}) \cdot (\frac{1}{N} \varsigma) = -\lambda + \frac{1}{N} w^{-1}(\varsigma),
\]
so that for $\alpha \in \fR_{\mathrm{s}}$ we have
\[
\textstyle
\langle \alpha, (t_{-\lambda} w^{-1})(\frac{1}{N} \varsigma) \rangle = -\langle \alpha, \lambda \rangle + \frac{1}{N} \langle w(\alpha), \varsigma \rangle.
\]
On the right-hand side we have $\langle w(\alpha), \varsigma \rangle >0$ if $w(\alpha) \in \fR_+$, and $\langle w(\alpha), \varsigma \rangle <0$ if $w(\alpha) \in -\fR_+$. This implies that the left-hand side lies between $0$ and $1$ iff $\langle \alpha,\lambda \rangle$ is $0$ in the first case, and $-1$ in the second case.
\end{proof}

\begin{lem}
\label{lem:length-res-dom}
For any $w \in W_\ext^S$ and $\mu \in -\bY_+$ we have
$\ell(w t_\mu)=\ell(t_\mu)+\ell(w)$.
\end{lem}

\begin{proof}
By~\eqref{eqn:WS-Wres} we can write $w=xt_\nu$ with $x \in W_\ext^\res$ and $\nu \in -\bY_+$.
Write $x= y t_\lambda$ with $\lambda \in \bY$ and $y \in W$. By~\eqref{eqn:formula-length}, for any $\eta \in -\bY_+$ we have
\[
\ell(xt_\eta) = \ell(yt_{\lambda+\eta}) = \sum_{\substack{\alpha \in \fR_+ \\ y(\alpha) \in \fR_+}} |\langle \alpha,\lambda+\eta \rangle| + \sum_{\substack{\alpha \in \fR_+ \\ y(\alpha) \in -\fR_+}} |1+\langle \alpha, \lambda+\eta \rangle|.
\]
By Lemma~\ref{lem:res-elements}, in the right-hand side we have $\langle \alpha, \lambda \rangle \leq 0$ for any $\alpha \in \fR_+$, and if moreover $y(\alpha) \in -\fR_+$ then at least one simple root $\gamma$ appearing in the decomposition of $\alpha$ as a sum of simple roots must satisfy $y(\gamma) \in -\fR_+$; we therefore have $\langle \alpha, \lambda \rangle \leq -1$ in this case. Letting $\rho$ denote one-half the sum of the positive roots, we see that
\[
\ell(xt_\eta) = -\langle 2\rho, \lambda+\eta \rangle - \ell(y).
\]
Comparing these formulas for $\eta=\nu$ and $\eta=\nu+\mu$, and using the fact that $\ell(t_\mu)=-\langle 2\rho,\mu \rangle$, we deduce the desired formula.
\end{proof}

%-------------------------------------------------------------------
\subsection{More on coset representatives}
\label{ss:more-coset}
%-------------------------------------------------------------------

We fix a finitary subset $A \subset S_\aff$, and consider the interaction between restricted elements and elements satisfying the conditions in Lemma~\ref{lem:double-min}.

\begin{lem}
\label{lem:double-min-antidom}
Let $y \in W_\ext^\res$ and $\lambda \in -\bY_+$. Then $y \in {}^A W^S_\ext$ iff $y t_\lambda \in {}^A W^S_\ext$.
\end{lem}

\begin{proof}
Assume that $y \in {}^A W^S_\ext$, i.e.~that $y$ is minimal in $W_A y$ and $w_A y \in W_\ext^S$. Then
\[
\ell(w_A y t_\lambda) = \ell(w_A y) + \ell(t_\lambda) = \ell(w_A) + \ell(y) + \ell(t_\lambda) = \ell(w_A) + \ell(yt_\lambda)
\]
where the first and last equalities use Lemma~\ref{lem:length-res-dom}. Hence $y t_\lambda$ is minimal in $W_A yt_\lambda$ by~\eqref{eqn:min-length}. On the other hand $w_A y t_\lambda = (w_A y) t_\lambda$ belongs to $W_\ext^S$ since $W_\ext^S$ is stable under right multiplication by elements of $-\bY_+$; hence $yt_\lambda \in {}^A W^S_\ext$.

Assume now that $yt_\lambda \in {}^A W^S_\ext$. We have
\begin{multline*}
\ell(w_A y w_\circ t_{w_\circ(\lambda)}) = \ell(w_A y t_\lambda w_\circ) = \ell(w_A) + \ell(yt_\lambda) + \ell(w_\circ) \\
= \ell(w_A) + \ell(y) + \ell(t_\lambda) + \ell(w_\circ).
\end{multline*}
On the other hand we have
\[
\ell(w_A y w_\circ t_{w_\circ(\lambda)}) \leq \ell(w_A y w_\circ) + \ell(t_{w_\circ(\lambda)}) \leq \ell(w_A) + \ell(y) + \ell(w_\circ) + \ell(t_{w_\circ(\lambda)}),
\]
and $\ell(t_\lambda) = \ell(t_{w_\circ(\lambda)})$.
 Thus, these inequalities must be equalities, showing in particular that $\ell(w_A y w_\circ) = \ell(w_A) + \ell(y) + \ell(w_\circ)$, and hence that $y \in {}^A W^S_\ext$.
\end{proof}

If we set
${}^A W^\res_\ext := {}^A W^S_\ext \cap W_\ext^\res$,
then by~\eqref{eqn:WS-Wres} and Lemma~\ref{lem:double-min-antidom} we have
\begin{equation}
\label{eqn:WS-Wres-Whit}
{}^A W^S_\ext = \{wt_\lambda : w \in {}^A W^\res_\ext, \, \lambda \in -\bY_+ \}.
\end{equation}

%%%%%%%%%%%%%%%%%%%%%%%%%%%%%%%%%
\section{Whittaker-type perverse sheaves on affine Grassmannians and affine flag varieties}
\label{sec:perverse}
%%%%%%%%%%%%%%%%%%%%%%%%%%%%%%%%%

%-------------------------------------------------------------------
\subsection{Affine Grassmannian and affine flag variety}
%-------------------------------------------------------------------

We now denote by $z$ an indeterminate, and consider the functor $\cL G$, resp.~$\cL^+ G$, from $\F$-algebras to groups, which sends $R$ to $G(R( \hspace{-1pt} (z) \hspace{-1pt} ))$, resp.~$G(R[ \hspace{-1pt} [z] \hspace{-1pt} ])$. It is well known (see e.g.~\cite{richarz}) that $\cL G$ is represented by a group ind-scheme over $\F$, and that $\cL^+ G$ is represented by a group scheme over $\F$. Moreover, the fppf quotient $(\cL G/\cL^+G)_{\mathrm{fppf}}$ is represented by an ind-projective ind-scheme, which is denoted $\Gr$ and called the \emph{affine Grassmannian} of $G$.

There is an obvious morphism of group schemes $\cL^+G \to G$ induced by the assignment $z \mapsto 0$.  Let $I \subset \cL^+G$ and $I_\unip \subset I$ be the preimages under this map of the Borel subgroup $B \subset G$ and its unipotent radical $U \subset B$, respectively. These are both subgroup schemes of $\cL^+G$.  The group $I$ is known as an \emph{Iwahori subgroup}, and $I_\unip$ as its pro-unipotent radical. 

We will consider also the affine flag variety $\Fl$ of $G$, defined as the fppf quotient $(\cL G/I)_{\mathrm{fppf}}$. Again $\Fl$ is represented by an ind-projective ind-scheme, and the natural morphism $\pi : \Fl \to \Gr$ is a Zariski locally trivial fibration with fibers isomorphic to $G/B$. 

Let $\Norm_G(T)$ be the normalizer of the maximal torus $T \subset G$, so that $\Norm_G(T)/T = W$.  For each $w \in W$, choose a representative $\dot w \in \Norm_G(T)$.  More generally, if $w \in W_\ext$, say $w = vt_\lambda$ with $v \in W$ and $\lambda \in \bY$, we set
\[
\dot w = \dot v z^\lambda \quad \in \cL G(\F).
\]

For $w \in W_\ext$ we will denote by $\Fl_w$ the $I$-orbit of the image of $\dot{w}$ in $\Fl$; then it is well known that $\Fl_w$ is also the $I_\unip$-orbit of the image of $\dot{w}$, that it is isomorphic to an affine space of dimension $\ell(w)$, and that we have
\[
\Fl_{\mathrm{red}} = \bigsqcup_{w \in W_\ext} \Fl_w \quad \text{and} \quad
%\]
%and for any $w,y \in W_\ext$ we have
%\[
\bigl( \overline{\Fl_w} \subset \overline{\Fl_y} \quad \Leftrightarrow \quad w \leq y \bigr).
\]

Similarly, for $w \in W_\ext^S$ we will denote by $\Gr_w$ the $I$-orbit of the image of $\dot{w}$ in $\Gr$. It is well known that $\Gr_w$ is also the $I_\unip$-orbit of the image of $\dot{w}$, that it is isomorphic to an affine space of dimension $\ell(w)$, and that we have
\[
\Gr_{\mathrm{red}} = \bigsqcup_{w \in W_\ext} \Gr_w
%\]
\quad \text{and} \quad
%\[
\pi^{-1}(\Gr_w)=\bigsqcup_{v \in W} \Fl_{wv}.
\]
The closure inclusion partial order on the set of $I$-orbits on $\Gr$ is governed by the restriction of the Bruhat order to $W_\ext^S$ (see~\S\ref{ss:coset}), i.e.~for $w,y \in W_\ext^S$ we have
\[
\overline{\Gr_w} \subset \overline{\Gr_y} \quad \Leftrightarrow \quad w \leq y.
\]

\begin{rmk}
It is common to label $I$-orbits on $\Gr$ by elements of $\bY$; compared with the labelling chosen here, the orbit usually associated with $\lambda$ is $\Gr_{\minR_\lambda}$ where we use the notation of~\S\ref{ss:coset}.
\end{rmk}

%-------------------------------------------------------------------
\subsection{Categories of \texorpdfstring{$I_\unip$}{Iu}-equivariant sheaves}
\label{ss:cat-Iequ-sheaves}
%-------------------------------------------------------------------

We now consider a prime number $\ell$ which is invertible in $\F$. We will consider fields $\bk$ which fall into one of the following two classes:
\begin{enumerate}
\item
$\bk$ is either a finite extension or an algebraic closure of $\mathbb{Q}_\ell$;
\item
$\bk$ is either a finite extension or an algebraic closure of $\mathbb{F}_\ell$.
\end{enumerate}
(When we need to distinguish these two cases, we will loosely say that $\bk$ has characteristic $0$ or $\bk$ has positive characteristic.)
In these settings we can consider the $I_\unip$-equivariant derived categories $\Db_{I_\unip}(\Gr,\bk)$ and $\Db_{I_\unip}(\Fl,\bk)$ of \'etale $\bk$-sheaves on $\Gr$ and $\Fl$ respectively. (More specifically, the case when $\bk$ is a finite extension of $\mathbb{Q}_\ell$ or $\mathbb{F}_\ell$ is classical,\footnote{Since $I_\unip$ is not of finite type and $\Gr,\Fl$ are ind-schemes rather than schemes, the definition of these categories requires a little bit of care, but is standard; we will not review these details here. Similar comments apply to various other equivariant derived categories considered below.} see~\cite{bbd}, and the case of algebraic closures is deduced using a colimit construction.) These categories have natural perverse t-structures, whose hearts will be denoted $\Perv_{I_\unip}(\Gr,\bk)$ and $\Perv_{I_\unip}(\Fl,\bk)$ respectively.

For any $w \in W_\ext$ we have a ``standard perverse sheaf'' $\DFl_w$ in $\Perv_{I_\unip}(\Fl,\bk)$, defined as the $!$-pushforward of the complex $\underline{\bk}_{\Fl_w}[\ell(w)]$ under the embedding $\Fl_w \to \Fl$, and a ``costandard perverse sheaf'' $\NFl_w$ in $\Perv_{I_\unip}(\Fl,\bk)$, defined as the $*$-pushforward of the complex $\underline{\bk}_{\Fl_w}[\ell(w)]$ under the embedding $\Fl_w \to \Fl$. (These complexes are indeed perverse sheaves since this embedding is affine.) The image of the unique (up to scalar) nonzero morphism $\DFl_w \to \NFl_w$ is simple, and will be denoted $\LFl_w$; it is the intersection cohomology complex associated with the constant local system on ${\Fl_w}$. Then the objects $(\LFl_w : w \in W_\ext)$ are representatives for the isomorphism classes of simple objects in the abelian category $\Perv_{I_\unip}(\Fl,\bk)$.

Similarly, for $w \in W_\ext^S$ we have a ``standard perverse sheaf'' $\DGr_w$ in $\Perv_{I_\unip}(\Gr,\bk)$, defined as the $!$-pushforward of the complex $\underline{\bk}_{\Gr_w}[\ell(w)]$ under the embedding $\Gr_w \to \Gr$, and a ``costandard perverse sheaf'' $\NGr_w$ in $\Perv_{I_\unip}(\Gr,\bk)$, defined as the $*$-pushforward of the complex $\underline{\bk}_{\Gr_w}[\ell(w)]$ under the embedding $\Gr_w \to \Gr$. (Once again, these complexes are indeed perverse sheaves.)
The image of the unique (up to scalar) nonzero morphism $\DGr_w \to \NGr_w$ is simple, and will be denoted $\LGr_w$; it is the intersection cohomology complex associated with the constant local system on ${\Gr_w}$. Then the objects $(\LGr_w : w \in W^S_\ext)$ are representatives for the isomorphism classes of simple objects in the abelian category $\Perv_{I_\unip}(\Gr,\bk)$.

Since the morphism $\pi : \Fl \to \Gr$ is smooth with connected fibers, the functor
\[
\pi^\dag:=\pi^*[\dim(G/B)] \cong \pi^![-\dim(G/B)] : \Db_{I_\unip}(\Gr,\bk) \to \Db_{I_\unip}(\Fl,\bk)
\]
is t-exact for the perverse t-structures, its restriction to perverse sheaves is fully faithful, and it sends simple perverse sheaves to simple perverse sheaves, see~\cite[Proposition~4.2.5]{bbd}; more explicitly, in this case we have
\begin{equation}
\label{eqn:pidag-IC}
\pi^\dag \LGr_w \cong \LFl_{ww_\circ}
\end{equation}
for any $w \in W_\ext^S$. 

The results of~\cite[\S 3.3]{bgs} show that the category $\Perv_{I_\unip}(\Fl,\bk)$ admits a natural structure of a highest weight category (in the sense of~\cite[\S 7]{riche-hab}) with weight poset $(W_\ext, \leq)$, standard objects the standard perverse sheaves $(\DFl_w : w \in W_\ext)$, and costandard objects the costandard perverse sheaves $(\NFl_w : w \in W_\ext)$. 
Similar comments apply to the category $\Perv_{I_\unip}(\Gr,\bk)$ (where the weight poset is now $W_\ext^S$, equipped with the restriction of the Bruhat order, and $\DFl_w, \NFl_w$ are replaced by $\DGr_w,\NGr_w$).

We will also occasionally consider the $I$-equivariant derived categories $\Db_I(\Fl,\bk)$ and $\Db_I(\Gr,\bk)$. We have forgetful functors
\[
\For^I_{I_\unip} : \Db_I(\Fl,\bk) \to \Db_{I_\unip}(\Fl,\bk), \quad \For^I_{I_\unip} : \Db_I(\Gr,\bk) \to \Db_{I_\unip}(\Gr,\bk),
\]
and the objects $\DFl_w, \NFl_w$ and $\DGr_w,\NGr_w$ naturally ``lift" to objects of $\Db_I(\Fl,\bk)$ and $\Db_I(\Gr,\bk)$ respectively (which will be denoted by the same symbol). We also have ``convolution" bifunctors
\begin{gather*}
\Db_I(\Fl,\bk) \times \Db_I(\Fl,\bk) \to \Db_I(\Fl,\bk), \quad \Db_I(\Fl,\bk) \times \Db_I(\Gr,\bk) \to \Db_I(\Gr,\bk), \\
\Db_{I_\unip}(\Fl,\bk) \times \Db_I(\Fl,\bk) \to \Db_{I_\unip}(\Fl,\bk), \quad \Db_{I_\unip}(\Fl,\bk) \times \Db_I(\Gr,\bk) \to \Db_{I_\unip}(\Gr,\bk),
\end{gather*}
which will all be denoted $\star^I$, and are compatible in all the expected ways. 

%-------------------------------------------------------------------
\subsection{Relation with the Satake category}
\label{ss:Satake-category}
%-------------------------------------------------------------------

Below we will also consider the $\cL^+G$-equiva\-riant derived category $\Db_{\cL^+ G}(\Gr,\bk)$. Once again this category has a natural perverse t-structure, whose heart will be denoted $\Perv_{\cL^+ G}(\Gr,\bk)$. For $\lambda \in \bY_+$ we will denote by $L_\lambda$ the image of $z^\lambda$ in $\Gr$, and by $\Gr^\lambda$ its $\cL^+G$-orbit; then we have
\[
\Gr^\lambda = \bigsqcup_{\mu \in W(\lambda)} \Gr_{\minR_\mu}
\quad
\text{and}
\quad
\Gr_{\mathrm{red}} = \bigsqcup_{\lambda \in \bY_+} \Gr^\lambda.
\]
The closure partial order on the set of $\cL^+G$-orbits on $\Gr$ is determined by the restriction of the Bruhat order to the set of elements $w \in W_\ext$ that are minimal in $WwW$: see Remark~\ref{rmk:order}. More explicitly, for $\lambda \in \bY_+$, the maximal element in $W t_\lambda W$ is $w_\circ t_\lambda$, so that for $\lambda,\mu \in \bY_+$ we have
\[
\overline{\Gr^\lambda} \subset \overline{\Gr^\mu} \quad \Leftrightarrow \quad w_\circ t_\lambda \leq w_\circ t_\mu.
\]
(It is a standard fact that this condition is also equivalent to the property that $\mu-\lambda$ is a sum of positive coroots.)

The simple objects in the category $\Perv_{\cL^+ G}(\Gr,\bk)$ are in natural bijection with $\bY_+$, via the operation sending $\lambda$ to the intersection cohomology complex $\IC^\lambda$ associated with the constant local system on $\Gr^\lambda$. The forgetful functor
\[
\For^{\cL^+G}_{I_\unip} : \Db_{\cL^+ G}(\Gr,\bk) \to \Db_{I_\unip}(\Gr,\bk)
\]
is t-exact, restricts to a fully faithful functor on perverse sheaves, and satisfies
\[
\For^{\cL^+G}_{I_\unip}(\IC^\lambda)=\LGr_{t_{w_\circ(\lambda)}}
\]
for any $\lambda \in \bY_+$.

To each $\lambda \in \bY_+$ one can also associate the ``standard'' and ``costandard'' objects defined respectively by
\[
\cI_!^\mu = {}^{\mathrm{p}}\hspace{-1pt} \tau^{\geq 0}(j^\mu_! \underline{\bk}_{\Gr^\mu}[\langle 2\rho,\mu \rangle]), \quad \cI_*^\mu = {}^{\mathrm{p}}\hspace{-1pt} \tau^{\leq 0}(j^\mu_* \underline{\bk}_{\Gr^\mu}[\langle 2\rho,\mu \rangle]), 
\]
where $j^\mu: \Gr^\mu \hookrightarrow \Gr$ is the inclusion and ${}^{\mathrm{p}} \hspace{-1pt} \tau^{\geq 0}, {}^{\mathrm{p}} \hspace{-1pt} \tau^{\leq 0}$ are the perverse truncation functors. With this notation there exists (up to scalar) a unique nonzero morphism $\cI_!^\mu \to \cI_*^\mu$, and its image is $\IC^\mu$. Once again the category $\Perv_{\cL^+ G}(\Gr,\bk)$ has a highest weight structure with standard objects the perverse sheaves $(\cI_!^\mu : \mu \in \bY_+)$ and costandard objects the perverse sheaves $(\cI_*^\mu : \mu \in \bY_+)$, see~\cite[Proposition~1.12.4]{bar}. (Contrary to the case of $I_\unip$-equivariant perverse sheaves, 
the proof of this claim relies on some subtle results on the geometry of $\cL^+G$-orbits on $\Gr$ due to Mirkovi{\'c}--Vilonen.)

As in the $I$-equivariant setting (see~\S\ref{ss:cat-Iequ-sheaves}), we have a convolution product
\begin{equation}
\label{eqn:convol-L+G}
\star^{\cL^+G} : \Db_{\cL^+ G}(\Gr,\bk) \times \Db_{\cL^+ G}(\Gr,\bk) \to \Db_{\cL^+ G}(\Gr,\bk)
\end{equation}
which equips $\Db_{\cL^+ G}(\Gr,\bk)$ with the structure of a monoidal category. In this case it is known that this product is t-exact (i.e., a product of perverse sheaves is perverse), and hence induces a monoidal structure on the abelian category $\Perv_{\cL^+G}(\Gr,\bk)$; see~\cite[\S 1.6.3]{bar} for details. The \emph{geometric Satake equivalence} describes the monoidal category $(\Perv_{\cL^+G}(\Gr,\bk),\star^{\cL^+G})$ in representation-theoretic terms: more explicitly, in~\cite{mv} the authors construct a canonical affine $\bk$-group scheme $G^\vee_\bk$ equipped with a split maximal torus $T^\vee_\bk$ whose group of characters is $\bY$ and a canonical equivalence of monoidal categories
\[
\Satake: (\Perv_{\cL^+G}(\Gr,\bk), \star^{\cL^+G}) \simto ( \Rep(G^\vee_\bk), \otimes).
\]
They also show that $G^\vee_\bk$ is a split connected reductive group over $\bk$, and that the root datum of $(G^\vee_\bk, T^\vee_\bk)$ is dual to that of $(G,T)$. Under this equivalence $\cI_!^\mu$, resp.~$\cI_*^\mu$, corresponds to the Weyl, resp.~induced, module of highest weight $\mu$.

%-------------------------------------------------------------------
\subsection{Root subgroups and unipotent subgroups}
\label{ss:root-subgps}
%-------------------------------------------------------------------

Recall (see e.g.~\cite[\S I.1.3]{jantzen}) that for each root $\alpha \in \fR_+$ there is a homomorphism
\[
\varphi_\alpha: \mathrm{SL}_2 \to G
\]
such that for $t \in T$, $x \in \F$ and $y \in \F^\times$ we have
\[
t \varphi_\alpha( \begin{smallmatrix} 1 & x \\ 0 & 1 \end{smallmatrix}) t^{-1} = \varphi_\alpha( \begin{smallmatrix} 1 & \alpha(t)x \\ 0 & 1 \end{smallmatrix}), \quad t \varphi_\alpha( \begin{smallmatrix} 1 & 0 \\ x & 1 \end{smallmatrix}) t^{-1} = \varphi_\alpha( \begin{smallmatrix} 1 & 0 \\ \alpha(t)^{-1} x & 1 \end{smallmatrix}), \quad \varphi_\alpha(\begin{smallmatrix} y & 0 \\ 0 & y^{-1} \end{smallmatrix})=\alpha^\vee(y).
\]
The image of the map $\Ga \to G$ given by $x \mapsto \varphi_\alpha( \begin{smallmatrix} 1 & x \\ 0 & 1 \end{smallmatrix})$ is often denoted $U_\alpha$, and called the \emph{root subgroup} of $G$ associated with $\alpha$.

We will now explain how to define certain (positive, simple) root subgroups of $\cL G$, attached to elements $s \in S_\aff$.  (For a discussion of more general root subgroups of $\cL G$, see~\cite[\S 3]{faltings}.)
First, if $s \in S$, let $\alpha_s \in \fR_{\mathrm{s}}$ be the corresponding simple root, and let
\[
U^+_s := \text{image of $U_{\alpha_s}$ under the natural map $G \to \cL G$.}
\]
On the other hand, if $s \in S_\aff \smallsetminus S$, then recall that $s = t_{\beta^\vee} s_{\beta^\vee}$ for a maximal short coroot $\beta^\vee \in \fR^\vee$, corresponding to a maximal (long) root $\beta \in \fR_+$.  In this case, define
\[
U^+_s := \text{image of the map $\Ga \to \cL G$ given by $x \mapsto \varphi_\beta(\begin{smallmatrix} 1 & 0 \\ z^{-1}x & 1\end{smallmatrix})$.}
\]
The construction above gives us an isomorphism $\Ga \cong U^+_s$ for each $s \in S_\aff$. (This isomorphism is not canonical, but is fixed once and for all.)

A direct calculation (cf.~\cite[\S 3]{faltings}) shows that for any $s \in S_\aff$, the group $\dot{s} I_\unip \dot{s}^{-1} \cap I_\unip$ is normal in $\dot{s} I_\unip \dot{s}^{-1}$, and that multiplication induces an isomorphism
\[
 U^+_s \ltimes (\dot{s} I_\unip \dot{s}^{-1} \cap I_\unip) \simto \dot{s} I_\unip \dot{s}^{-1}.
\]
This identification gives rise to a quotient map
\[
\psi_s: \dot{s} I_\unip \dot{s}^{-1} \to U^+_s \cong \Ga.
\]

More generally, consider a finitary subset $A \subset S_\aff$. 
Then there is a parahoric group scheme $P_A \subset \cL G$ such that
\[
P_A(\F) = \bigcup_{w \in W_A} I(\F) \dot w I(\F).
\]
If we set $I_\unip^A := \dot{w}_A I_\unip \dot{w}_A^{-1}$,
the intersection $I_\unip^{A} \cap I_\unip$ is the pro-unipotent radical of $P_A$, and the quotient $P_A/I_\unip^{A} \cap I_\unip$ is a reductive algebraic group $M_A$ over $\F$ whose Weyl group is $W_A$. If we set $U^+_A:=I^{A}_\unip/I^{A}_\unip \cap I_\unip \subset M_A$, then $U^+_A$ is the unipotent radical of a (positive) Borel subgroup of $M_A$.  There is a canonical isomorphism
\[
U^+_A/[U^+_A,U^+_A] \cong \prod_{s \in A} U^+_s,
\]
which we use to define the map $\psi_A: I^{A}_\unip \to \Ga$ 
as the composition
\[
I^{A}_\unip \to I^{A}_\unip/I^{A}_\unip \cap I_\unip = U^+_A \to U^+_A/[U^+_A,U^+_A] \cong \prod_{s \in A} U^+_s = \prod_{s \in A} \Ga \xrightarrow{+} \Ga.
\]

An important special case of this construction is when $A=S$ (so that $W_S=W$).
In this case we have $P_S=\cL^+G$, $I_\unip^{S}$ is the inverse image of the unipotent radical $U^+$ of the Borel subgroup of $G$ opposite to $B$ (with respect to $T$) under the evaluation morphism $\cL^+G \to G$, and the intersection $I^{S}_\unip \cap I_\unip$ is the kernel of this morphism.

%-------------------------------------------------------------------
\subsection{Whittaker categories}
\label{ss:Whit}
%-------------------------------------------------------------------

We assume from now on that $\F$ has characteristic $p>0$, and that $\bk$ contains a nontrivial $p$-th root of unity. This allows us to choose a nontrivial homomorphism $\Z/p\Z \to \bk^\times$, which in turn determines an Artin--Schreier local system on $\Ga$, denoted by $\AS$.

Let $A \subset S_\aff$ be a finitary subset.
We set
$\cX_A := \psi_A^*\AS$.
Using the techniques spelled out e.g.~in~\cite[Appendix~A]{modrap1} one can define the $(I_\unip^A,\cX_A)$-equivariant derived categories
\[
\Db_{(I_\unip^A,\cX_A)}(\Fl,\bk) \quad \text{and} \quad \Db_{(I_\unip^A,\cX_A)}(\Gr,\bk)
\]
of $\bk$-sheaves on $\Fl$ and $\Gr$ respectively.
These categories admit natural perverse t-structures, whose hearts will be denoted $\Perv_{(I_\unip^A,\cX_A)}(\Fl,\bk)$ and $\Perv_{(I_\unip^A,\cX_A)}(\Gr,\bk)$ respectively.

For $w \in W_\ext$ we will denote by $\Fl_w^A$ the $I_\unip^A$-orbit of the image of $\dot{w}$ in $\Fl$; then
\[
\Fl_{\mathrm{red}} = \bigsqcup_{w \in W_\ext} \Fl^A_w.
\]
By definition we have $\Fl^A_w = \dot{w}_A \cdot \Fl_{w_A w}$; it follows that for $y,w \in W_\ext$ we have
\[
\overline{\Fl^A_w} \subset \overline{\Fl^A_y} \quad \Leftrightarrow \quad w_A w \leq w_A y.
\]
Below we will mainly consider these orbits in the case where $w$ and $y$ are minimal in $W_A w$ and $W_A y$ respectively; 
in this case, in view of the discussion in~\S\ref{ss:coset} we have the simpler characterization
\[
\overline{\Fl^A_w} \subset \overline{\Fl^A_y} \quad \Leftrightarrow \quad w \leq y.
\]

It is a standard fact that the orbit $\Fl^A_w$ supports a nonzero $(I_\unip^A,\cX_A)$-equiva\-riant local system if and only if $w$ has minimal length in the coset $W_A w$; in this case there exists a unique such local system of rank $1$, and the corresponding standard, resp.~costandard, perverse sheaf (obtained by taking the $!$-pushforward, resp.~$*$-pushforward, of the shift by $\dim(\Fl^A_w)=\ell(w_A w)$ of this local system under the embedding $\Fl_w^A \to \Fl$) will be denoted $\DFl_w^A$, resp.~$\NFl_w^A$. Once again there exists a unique (up to scalar) nonzero morphism $\DFl_w^A \to \NFl_w^A$, whose image will be denoted $\LFl_w^A$, and the objects
\[
(\LFl_w^A : \text{$w \in W_\ext$ minimal in $W_A w$})
\]
are representatives for the isomorphism classes of simple objects in the abelian category $\Perv_{(I_\unip^A,\cX_A)}(\Fl,\bk)$. 

In this setting also the abelian category $\Perv_{(I_\unip^A,\cX_A)}(\Fl,\bk)$ has a natural structure of a highest weight category, with weight poset $\{w \in W_\ext \mid \text{$w$ minimal in $W_A w$}\}$ (equipped with the restriction of the Bruhat order).
A basic example of an element minimal in its coset in $W_A \backslash W_\ext$ is the identity element $e$. Since this element is minimal for the Bruhat order, the canonical morphism $\DFl^A_e \to \NFl^A_e$ is an isomorphism, and we have 
\begin{equation}
\label{eqn:DLN-Whit-e}
\DFl^A_e=\LFl^A_e=\NFl^A_e.
\end{equation}

These considerations have analogues for sheaves on $\Gr$, as follows.
For $w \in W^S_\ext$ we will denote by $\Gr_w^A$ the $I_\unip^A$-orbit of the image of $\dot{w}$ in $\Gr$. Then
\[
\Gr_{\mathrm{red}} = \bigsqcup_{w \in W^S_\ext} \Gr^A_w,
\]
and for $w,y \in W^S_\ext$ we have 
$\overline{\Gr^A_w} \subset \overline{\Gr^A_y}$ if and only if the maximal element in $w_A w W$ is smaller than the maximal element in $w_A y W$ (for the Bruhat order).  In the special case where $w,y \in {}^A W^S_\ext$, this condition is also equivalent to $w \leq y$.

It is a standard fact that the orbit $\Gr^A_w$ supports a nonzero $(I_\unip^A,\cX_A)$-equiva\-riant local system if and only if $w \in {}^A W^S_\ext$ (see~\cite[Appendix~A]{acr} for similar considerations).
In this case there exists a unique such local system of rank $1$, and the corresponding standard, resp.~costandard, perverse sheaf (obtained by taking the $!$-pushforward, resp.~$*$-pushforward, of the shift by $\dim(\Gr^A_w)=\ell(w_A w)$ of this local system under the embedding $\Gr_w^A \to \Gr$) will be denoted $\DGr_w^A$, resp.~$\NGr_w^A$. Once again there exists a unique (up to scalar) nonzero morphism $\DGr_w^A \to \NGr_w^A$, whose image will be denoted $\LGr_w^A$, and the objects
\[
(\LGr_w^A : w \in {}^A W^S_\ext)
\]
are representatives for the isomorphism classes of simple objects in the abelian category $\Perv_{(I_\unip^A,\cX_A)}(\Gr,\bk)$. The standard and costandard objects defined above also endow $\Perv_{(I_\unip^A,\cX_A)}(\Gr,\bk)$ with a natural structure of a highest weight category with weight poset ${}^A W_\ext^S$ (with respect to the restriction of the Bruhat order). 

In this setting again we have a t-exact functor
\[
\pi^\dag : \Db_{(I_\unip^A,\cX_A)}(\Gr,\bk) \to \Db_{(I_\unip^A,\cX_A)}(\Fl,\bk),
\]
which restricts to a fully faithful functor on perverse sheaves and satisfies
\begin{equation}
\label{eqn:pidag-IC-A}
\pi^\dag(\LGr^A_w) \cong \LFl^A_{w w_\circ}
\end{equation}
if $w \in {}^AW_\ext^S$. We also have natural functors
\[
\pi_*, \pi_! : \Db_{(I_\unip^A,\cX_A)}(\Fl,\bk) \to \Db_{(I_\unip^A,\cX_A)}(\Gr,\bk).
\]

%-------------------------------------------------------------------
\subsection{Averaging functors}
\label{ss:Av}
%-------------------------------------------------------------------

Of course $I_\unip^A$ contains $I_\unip^A \cap I_\unip$, and by construction the restriction of $\cX_A$ to this subgroup is trivial. We therefore have a canonical forgetful functor
\[
\For_A: \Db_{(I_\unip^A,\cX_A)}(\Fl,\bk) \to \Db_{I_\unip^A \cap I_\unip}(\Fl,\bk),
\]
where the right-hand side is the $(I_\unip^A \cap I_\unip)$-equivariant derived category of $\bk$-sheaves on $\Fl$. This functor
is fully faithful, and the techniques of~\cite[Appendix A]{modrap1} show that it admits left and right adjoints, denoted by
\[
\av^A_{\psi,!}, \av^A_{\psi,*}: \Db_{I_\unip^A \cap I_\unip}(\Fl,\bk) \to \Db_{(I_\unip^A,\cX_A)}(\Fl,\bk)
\]
respectively; we have
\[
\av^A_{\psi,!}(\cF) = (\mathrm{act}_A)_!(\cX_A \, \widetilde{\boxtimes} \, \cF)[2\dim(U^+_A)], \quad \av^A_{\psi,*}(\cF) = (\mathrm{act}_A)_*(\cX_A \, \widetilde{\boxtimes} \, \cF)
\]
where 
$\mathrm{act}_A : I_\unip^A \times^{I_\unip^A \cap I_\unip} \Fl \to \Fl$ is the action morphism and $\cX_A \, \widetilde{\boxtimes} \, \cF$ is the unique complex whose pullback to $I_\unip^A \times \Fl$ is $\cX_A \boxtimes \cF$. Similarly we have a forgetful functor
\[
\For'_A: \Db_{I_\unip}(\Fl,\bk) \to \Db_{I_\unip^A \cap I_\unip}(\Fl,\bk),
\]
which admits left and right adjoints denoted by
\[
\av^A_{!}, \av^A_{*}: \Db_{I_\unip^A \cap I_\unip}(\Fl,\bk) \to \Db_{I_\unip}(\Fl,\bk),
\]
and defined by formulas similar to those above (involving the constant local system instead of $\cX_A$).

We will set
\begin{gather*}
\Av^A_{\psi,!} := \av^A_{\psi,!} \circ \For'_A[-\dim(U^+_A)], \quad \Av^A_{\psi,*} := \av^A_{\psi,*} \circ \For'_A[\dim(U^+_A)],\\
\Av^A_{!} := \av^A_{!} \circ \For_A[-\dim(U^+_A)], \quad \Av^A_{*} := \av^A_{*} \circ \For_A[\dim(U^+_A)];
\end{gather*}
then we have adjoint pairs $(\Av^A_{\psi,!}, \Av^A_*)$ and $(\Av^A_{!}, \Av^A_{\psi,*})$.

Similar considerations apply to sheaves on $\Gr$; we will use the same notation for the corresponding functors relating the categories $\Db_{I_\unip}(\Gr,\bk)$ and $\Db_{(I_\unip^A,\cX_A)}(\Gr,\bk)$. The base change theorem guarantees that we have canonical isomorphisms
\begin{equation}
\label{eqn:Av-pi}
\Av^A_{\psi,?} \circ \pi^\dag \cong \pi^\dag \circ \Av^A_{\psi,?}, \quad \Av^A_{?} \circ \pi^\dag \cong \pi^\dag \circ \Av^A_{?} \quad \text{for ${?}={!}$ or $*$.}
\end{equation}

%---------------------------------------------------------------------
\subsection{Study of Whittaker averaging functors}
\label{ss:Av-Whit}
%---------------------------------------------------------------------

The following claim is standard (see e.g.~\cite{bbm,by,abbgm}).

\begin{lem}
\label{lem:isom-Av}
For sheaves on $\Fl$ and $\Gr$,
there exists a canonical isomorphism of functors
\[
\Av^A_{\psi,!} \simto \Av^A_{\psi,*},
\]
and these functors are t-exact.
\end{lem}

\begin{proof}[Proof sketch]
To fix notation we consider sheaves on $\Fl$; the case of $\Gr$ is similar. 
The natural morphism of functors $(\mathrm{act}_A)_! \to (\mathrm{act}_A)_*$ induces a morphism of functors
\[
\av^A_{\psi,!}[-\dim(U_A^+)] \to \av^A_{\psi,*}[\dim(U_A^+)],
\]
from which we obtain a morphism $\Av^A_{\psi,!} \to \Av^A_{\psi,*}$. Since the category $\Db_{I_\unip}(\Fl,\bk)$ is generated (as a triangulated category) by the essential image of the forgetful functor $\For^I_{I_\unip} $ (see~\S\ref{ss:cat-Iequ-sheaves}), to show that this morphism is an isomorphism it suffices to do so for its composition with this functor. Now from the definitions we see that the compositions $\Av^A_{\psi,!} \circ \For^I_{I_\unip}$ and $\Av^A_{\psi,*} \circ \For^I_{I_\unip}$ can be described as convolution on the left with the objects $\DFl_e^A$ and $\NFl_e^A$ respectively. Since these objects are canonically isomorphic (see~\eqref{eqn:DLN-Whit-e}), we deduce the desired isomorphism.

Since the functor $\Av^A_{\psi,!}$,
resp.~$\Av^A_{\psi,*}$, 
is defined in terms of a $!$-pushforward, resp.~$*$-pushforward, along an affine morphism, it is left t-exact, resp.~right t-exact, by~\cite[Th\'eor\`eme 4.1.1, Corollaire~4.1.2]{bbd}. Since these functors are isomorphic, they are therefore t-exact.
\end{proof}

In view of Lemma~\ref{lem:isom-Av}, the functors
$\Av^A_{\psi,!}$ and $\Av^A_{\psi,*}$ will be identified below, and denoted simply by $\Av^A_{\psi}$. This lemma implies in particular that we have canonical isomorphisms
\begin{equation}
\label{eqn:Av-pi-2}
\Av^A_{\psi} \circ \pi_* \cong \pi_* \circ \Av^A_{\psi}, \quad \Av^A_{\psi} \circ \pi_! \cong \pi_! \circ \Av^A_{\psi}
\end{equation}
The behavior of $\Av^A_{\psi}$ on our ``special'' perverse sheaves is described as follows.

\begin{lem}
\phantomsection
\label{lem:AvWhit}
\begin{enumerate}
\item
\label{it:AvWhit-1}
If $w \in W_\ext$ is minimal in $W_A w$, then for $y \in W_A$ we have 
\[
\Av^A_{\psi}(\DFl_{yw}) \cong \DFl^A_w, \quad \Av^A_{\psi}(\NFl_{yw}) \cong \NFl^A_w.
\]
\item
\label{it:AvWhit-2}
If $w \in W_{\ext}$ is minimal in $W_A w$, then for $y \in W_A$ the object
$\Av^A_\psi(\LFl_{yw})$ is isomorphic to
$\LFl^A_w$ if $y=e$, and vanishes otherwise.
\item
\label{it:AvWhit-3}
Let $w \in W_\ext^S$, and write $w=yx$ with $y \in W_A$ and $x$ minimal in $W_A x$. Then $x \in W_\ext^S$, and we have
\[
\Av^A_{\psi}(\DGr_{w}) \cong \begin{cases} \DGr^A_x & \text{if $x \in {}^A W_\ext^S$;} \\
0 & \text{otherwise} \end{cases},
\quad
\Av^A_{\psi}(\NGr_{w}) \cong \begin{cases} \NGr^A_x & \text{if $x \in {}^A W_\ext^S$;} \\
0 & \text{otherwise.} \end{cases}
\]
\item
\label{it:AvWhit-4}
Let $w \in W_\ext^S$. The object
$\Av^A_\psi(\LGr_{w})$ is isomorphic to
$\LGr^A_w$ if $w \in {}^A W_\ext^S$, and vanishes otherwise.
\end{enumerate}
\end{lem}

\begin{proof}
For~\eqref{it:AvWhit-1}--\eqref{it:AvWhit-2}, the proof can be adapted from those of~\cite[Lemmas~4.4.6 \&~4.4.8]{by}.

\eqref{it:AvWhit-3} The fact that $x \in W_\ext^S$ follows from Lemma~\ref{lem:minimal-mult-Saff}.
For the description of $\Av^A_{\psi}(\DGr_{w})$ and $\Av^A_{\psi}(\NGr_{w})$, the same arguments as for~\eqref{it:AvWhit-1} reduce the proof to the case $y=e$, i.e.~$w$ is minimal in $W_A w W$. In this case we observe that
\[
\Av^A_{\psi}(\DGr_{w}) \cong \Av^A_{\psi}(\pi_! \DFl_{w}) \overset{\eqref{eqn:Av-pi-2}}{\cong} \pi_! \Av^A_{\psi}(\DFl_{w}) \overset{\eqref{it:AvWhit-1}}{\cong} \pi_! \DFl^A_w.
\]
Now $\pi_! \DFl^A_w$ is isomorphic to $\DGr^A_w$ if $w \in {}^A W_\ext^S$, and $0$ otherwise; see~\cite[Lemma~A.1]{acr} for a proof in the similar setting of Kac--Moody flag varieties. This proves the claim for $\Av^A_{\psi}(\DGr_{w})$; the case of $\Av^A_{\psi}(\NGr_{w})$ is similar.

\eqref{it:AvWhit-4}
We have
\[
\pi^\dag(\Av^A_\psi(\LGr_{w})) \overset{\eqref{eqn:Av-pi}}{\cong} \Av^A_\psi(\pi^\dag(\LGr_{w})) \overset{\eqref{eqn:pidag-IC}}{\cong} \Av^A_\psi(\LFl_{ww_\circ}).
\]
By~\eqref{it:AvWhit-2} the rightmost expression vanishes unless $ww_\circ$ is minimal in $W_A w w_\circ$, which by definition is equivalent to $w \in {}^A W_\ext^S$. In case $w$ belongs to ${}^A W_\ext^S$, the rightmost term is isomorphic to $\LFl^A_{ww_\circ}$. We have a simple perverse sheaf $\LGr^A_w$ on $\Gr$, and comparing the formula above with~\eqref{eqn:pidag-IC-A} we see that
$\pi^\dag(\Av^A_\psi(\LGr_{w})) \cong \pi^\dag(\LGr^A_w)$.
The desired claim follows, by full faithfulness of $\pi^\dag$ on perverse sheaves.
\end{proof}

%-------------------------------------------------------------------
\subsection{Study of Iwahori averaging functors}
\label{ss:Av-Iw}
%-------------------------------------------------------------------

We finish this section with some properties of the averaging functors $\Av^A_{!}$ and $\Av^A_{*}$.

\begin{lem}
\phantomsection
\label{lem:Av-Iw}
\begin{enumerate}
\item
\label{it:Av-Iw-1}
The functors $\Av^A_{!}$ and $\Av^A_{*}$ are t-exact.
\item
\label{it:Av-Iw-2}
There exists an isomorphism of functors
\[
\Av^A_! \circ \Av^A_\psi \circ \For^{I}_{I_\unip} \cong \Av^A_!(\DFl^A_e) \star^I (-)
\]
which identifies the morphism $\Av^A_! \circ \Av^A_\psi \circ \For^{I}_{I_\unip} \to \For^{I}_{I_\unip}$ induced by adjunction with the morphism induced by a surjection $\Av^A_!(\DFl^A_e) \to \LFl_e$.
\end{enumerate}
\end{lem}

\begin{proof}
\eqref{it:Av-Iw-1}
The functor $\Av^A_{*}$ is the right adjoint of the exact functor $\Av^A_{\psi}$, so it is left exact. On the other hand this functor is defined in terms of $*$-pushforward along an affine morphism, so it is right exact by~\cite[Th\'eor\`eme 4.1.1]{bbd}. It is therefore exact. Dual arguments apply to $\Av^A_!$.

\eqref{it:Av-Iw-2} 
As explained in the course of the proof of Lemma~\ref{lem:isom-Av} the functor $\Av^A_\psi \circ \For^{I}_{I_\unip}$ identifies with $\DFl^A_e \star^I (-)$. The desired claims follow.
\end{proof}

%%%%%%%%%%%%%%%%%
\section{The geometric Steinberg formula}
\label{sec:geometric-Steinberg}
%%%%%%%%%%%%%%%%%

%------------------------------------------------------------------
\subsection{Statement}
\label{ss:statement-Steinberg}
%------------------------------------------------------------------

If $A \subset S_\aff$ is a finitary subset,
the same constructions as for~\eqref{eqn:convol-L+G} provide a convolution bifunctor
\[
\star^{\cL^+G} : \Db_{(I^A_\unip, \cX_A)}(\Gr,\bk) \times \Db_{\cL^+ G}(\Gr,\bk) \to \Db_{(I^A_\unip, \cX_A)}(\Gr,\bk).
\]
It is known that this bifunctor is again t-exact, in the sense that for $\cF$ in $\Perv_{(I^A_\unip, \cX_A)}(\Gr,\bk)$ and $\cG$ in $\Perv_{\cL^+ G}(\Gr,\bk)$ the convolution $\cF \star^{\cL^+G} \cG$ is perverse; see~\cite[Lemma~2.3]{bgmrr} for details and references.

Recall the subsets ${}^A W^S_\ext$ and ${}^A W^\res_\ext$ of $W_\ext$ introduced in~\S\ref{ss:coset} and~\S\ref{ss:more-coset} respectively. 
The following statement is the first main result of this paper, which gives a geometric counterpart of the Steinberg tensor product theorem for representations of reductive groups over fields of positive characteristic (or of quantum groups at a root of unity).

\begin{thm}
\label{thm:geometric-Steinberg}
Let $y \in {}^A W^\res_\ext$. 
Then for any $\mu \in \bY_+$ we have
\[
\LGr^A_y \star^{\cL^+G} \IC^\mu \cong \LGr^A_{yt_{w_\circ(\mu)}}.
\]
\end{thm}

\begin{rmk}
Note that by~\eqref{eqn:WS-Wres-Whit} the element $yt_{w_\circ(\mu)}$ belongs to ${}^A W^S_\ext$,
so it does indeed label a simple object in $\Perv_{(I_\unip^A,\cX_A)}(\Gr,\bk)$; in fact, this equality shows that any label for a simple object of $\Perv_{(I_\unip^A,\cX_A)}(\Gr,\bk)$ has the form $yt_{w_\circ(\mu)}$ for $y$ and $\mu$ as in Theorem~\ref{thm:geometric-Steinberg}. In other words, this theorem describes all the simple objects in $\Perv_{(I_\unip^A,\cX_A)}(\Gr,\bk)$ in terms of those whose label belongs to ${}^A W^\res_\ext$ and the simple objects in the Satake category $\Perv_{\cL^+ G}(\Gr,\bk)$.
\end{rmk}

In the course of the proof of this theorem we will also establish the following result, which is the second main result of the paper.

\begin{thm}
\label{thm:geometric-Steinberg-ff}
For any $y \in {}^A W_\ext^\res$, the functor
\[
\Phi_{y,A} := \LGr^A_y \star^{\cL^+G} (-) : \Perv_{\cL^+ G}(\Gr,\bk) \to \Perv_{(I^A_\unip,\cX_A)}(\Gr,\bk)
\]
is fully faithful.
\end{thm}

%------------------------------------------------------------------
\subsection{Preliminaries}
\label{ss:proof-steinberg}
%------------------------------------------------------------------

Our goal in this subsection is to prove the technical Lemma~\ref{lem:dimension}, which will be used crucially in the proofs of Theorems~\ref{thm:geometric-Steinberg} and~\ref{thm:geometric-Steinberg-ff}. This result will follows from
some claims (essentially taken from~\cite{abbgm}) on dimensions of certain subschemes of $\Gr$. The starting point for these proofs is a lemma from~\cite{fgv} which is closely related to the ``geometric Casselman--Shalika formula'' 
proved independently in~\cite{fgv} and~\cite{np}.

For the general theory of ind-schemes we refer to~\cite{richarz}.
For any $\mu \in \bY$ we have the ``semi-infinite orbits''
\[
\mathrm{S}_\mu, \mathrm{T}_\mu \subset \Gr,
\]
where we follow the conventions of~\cite{mv} or~\cite{bar}. (These are ind-schemes, endowed with natural morphisms $\mathrm{S}_\mu \to \Gr$, $\mathrm{T}_\mu \to \Gr$ which are representable by locally closed immersions.) 
Recall that $U$, resp.~$U^+$, is the unipotent radical of $B$, resp.~the Borel subgroup opposite to $B$ with respect to $T$; we have the loop group $\cL U$ associated with $U$, resp.~the loop group $\cL U^+$ associated with $U^+$, and
\[
\mathrm{T}_\mu(\F) = \cL U(\F) \cdot L_\mu, \quad \mathrm{S}_\mu(\F) = \cL U^+(\F) \cdot L_\mu.
\]
A crucial feature of these sub-ind-schemes is the fact that if $\lambda \in \bY_+$ and $\mu \in \bY$, then the intersection $\mathrm{T}_\mu \cap \Gr^\lambda$, resp.~$\mathrm{S}_\mu \cap \Gr^\lambda$, is empty unless $\lambda-\mathsf{dom}(\mu)$ is a sum of positive coroots, and that in this case this intersection is a scheme of finite type such that
\[
\dim(\Gr^\lambda \cap \mathrm{S}_\mu) = \langle \rho, \lambda + \mu \rangle, \quad \text{resp.} \quad \dim(\Gr^\lambda \cap \mathrm{T}_\mu) = \langle \rho, \lambda - \mu \rangle;
\]
see~\cite[Theorem~3.2]{mv} for the original reference and~\cite[Theorem~1.5.2]{br} for a more detailed treatment and further references.

Let us fix, for any $\alpha \in -\fRs$, an isomorphism $U_\alpha \cong \Ga$, where $U_\alpha \subset U$ is the root subgroup associated with $\alpha$ (i.e.~the image of the subgroup of lower-triangular unipotent matrices under the morphism $\varphi_{-\alpha}$ from~\S\ref{ss:root-subgps}). Then we obtain a group homomorphism
\[
\chi : U \to U/(U,U) \xleftarrow{\sim} \prod_{\alpha \in -\fRs} U_\alpha \cong \prod_{\alpha \in -\fRs} \Ga \xrightarrow{+} \Ga.
\]
We will also denote by $\chi^+ : U^+ \to \Ga$ the composition of $\chi$ with the isomorphism $U^+ \simto U$ given by $g \mapsto \dot{w}_\circ g \dot{w}_\circ^{-1}$, and denote by $\chi_{\cL U}$, resp.~$\chi_{\cL U^+}$, the composition
\[
\cL U \xrightarrow{\cL\chi} \cL\Ga \xrightarrow{\mathrm{res}} \Ga, \quad \text{resp.} \quad \cL U^+ \xrightarrow{\cL\chi^+} \cL\Ga \xrightarrow{\mathrm{res}} \Ga
\]
where the first morphism is the morphism on loop groups induced by $\chi$, resp.~$\chi^+$, and the second one is the ``residue'' morphism sending a Laurent series to the coefficient of $z^{-1}$ (see e.g.~\cite[\S 3.4]{bgmrr}). Then for any $\mu \in \bY$ there exist unique morphisms $\chi^{\mathrm{T}}_\mu : \mathrm{T}_\mu \to \Ga$ and $\chi^{\mathrm{S}}_\mu : \mathrm{S}_\mu \to \Ga$ such that $\chi_\mu^{\mathrm{T}}(u \cdot L_\mu)=\chi_{\cL U}(z^{-\mu} u z^{\mu})$ for any $u \in \cL U$ and $\chi_\mu^{\mathrm{S}}(u \cdot L_\mu)=\chi_{\cL U^+}(z^{-\mu} u z^{\mu})$ for any $u \in \cL U^+$.

The starting point for our proofs is the following claim, taken from~\cite[Lemma~7.1.7]{fgv}.

\begin{lem}
\label{lem:fgv}
Let $\mu \in \bY$ and $\lambda \in \bY_+$ be such that $\mu \neq w_\circ(\lambda)$. Then the restriction of $\chi_\mu^{\mathrm{S}}$ to any irreducible component of $\mathrm{S}_\mu \cap \Gr^\lambda$ is dominant.
\end{lem}

We deduce the following.

\begin{lem}
\label{lem:geom-cas-sha}
For any $\mu \in \bY$, the intersection $\mathrm{S}_\mu \cap \mathrm{T}_0 := \mathrm{S}_\mu \times_{\Gr} \mathrm{T}_0$ is a scheme of finite type, empty unless $\mu$ is a sum of positive coroots, and of dimension at most $\langle \rho,\mu \rangle$. If $\mu \neq 0$, then we have
\[
\dim \bigl( \mathrm{S}_\mu \cap (\chi_0^{\mathrm{T}})^{-1}(0) \bigr) < \langle \rho,\mu \rangle.
\]
\end{lem}

\begin{proof}
The fact that $\mathrm{S}_\mu \cap \mathrm{T}_0$ is a scheme of finite type is noted in~\cite[Proof of Proposition~6.4]{bfgm}. For any $\lambda \in \bY$, multiplication by $z^\lambda \dot{w}_\circ^{-1}$ induces an isomorphism
\[
\mathrm{S}_\mu \cap \mathrm{T}_0 \simto \mathrm{T}_{\lambda+w_\circ(\mu)} \cap \mathrm{S}_\lambda.
\]
Now by~\cite[Proposition~6.4]{bfgm}, if $\lambda$ is sufficiently far in the antidominant cone the right-hand side is contained in $\Gr^{w_\circ(\lambda)+\mu}$, and hence in $\Gr^{w_\circ(\lambda)+\mu} \cap \mathrm{S}_\lambda$, which as explained above is empty unless the coweight $w_\circ(\lambda)+\mu - w_\circ(\lambda)=\mu$
%\[
%\lambda - w_\circ(w_\circ(\lambda)+\mu) = \lambda - (\lambda+w_\circ(\mu)) = -w_\circ(\mu)
%\]
is a sum of positive coroots and has dimension 
\[
\langle \rho, w_\circ(\lambda)+\mu+\lambda \rangle = \langle \rho,\mu \rangle
\]
in this case.
Through this identification, the map $\chi_0^{\mathrm{T}}$ becomes the restriction of $\chi^{\mathrm{S}}_\lambda$. By Lemma~\ref{lem:fgv}, if $\mu \neq 0$ this map is nonconstant on any irreducible component of $\Gr^{w_\circ(\lambda)+\mu} \cap \mathrm{S}_\lambda$,
%of dimension $\langle \rho,\mu \rangle$, 
which implies that $\dim ( \mathrm{S}_\mu \cap (\chi_0^{\mathrm{T}})^{-1}(0) ) < \langle \rho,\mu \rangle$, as desired.
\end{proof}

\begin{cor}
\label{cor:dim-intersection}
Let $y \in W_\ext^S$, $\mu \in \bY_+$ and $\nu \in \bY$, 
and write $y=wt_\lambda$ with $w \in W$ and $\lambda \in \bY$. The intersection $(\dot{w} \mathrm{S}_{\nu}) \cap \Gr_y$ is empty unless $w_\circ(\lambda)-\mathsf{dom}(\nu)$ is a sum of positive coroots, and in this case we have
\[
\dim \bigl( (\dot{w} \mathrm{S}_{\nu}) \cap \Gr_y \bigr) \leq \langle \rho, \nu-\lambda \rangle.
\]
Moreover, if $y \in W_\ext^\res$ and $\nu \neq \lambda$ this inequality is strict.
\end{cor}

\begin{proof}
Note that the coweight $\lambda$ is necessarily antidominant here, by Lemma~\ref{lem:res-elements} and~\eqref{eqn:WS-Wres}. 
We have
\[
(\dot{w} \mathrm{S}_{\nu}) \cap \Gr_y = (\dot{w} \mathrm{S}_{\nu}) \cap (I \cdot \dot{w} L_\lambda) = \dot{w} \cdot \bigl( \mathrm{S}_\nu \cap (\dot{w}^{-1} I \dot{w} L_\lambda) \bigr) \subset \dot{w} \cdot \bigl( \mathrm{S}_\nu \cap \Gr^{w_\circ(\lambda)} \bigr).
\]
The usual properties of intersections of spherical orbits with semi-infinite orbits recalled at the beginning of the subsection show that the right-hand side is empty unless $w_\circ(\lambda)-\mathsf{dom}(\nu)$ is a sum of positive coroots, and has dimension $\langle \rho, w_\circ(\lambda)+\nu \rangle = \langle \rho, \nu-\lambda \rangle$ in this case. Our first claim follows.

Recall that for any $\beta \in \fR$ and $n \in \Z$ we have a ``root subgroup'' $U_{\beta,n} \subset \cL G$ defined as in~\cite[Proof of Lemma~3.10]{bgmrr}. If we set
\[
J := \prod_{\beta \in -\fR_+} \prod_{i=n_\beta}^{\langle \lambda,\beta \rangle-1} U_{\beta,i} \subset \dot{w}^{-1} I \dot{w} \quad
\text{where}
\quad
n_\beta = \begin{cases}
0 & \text{if $w(\beta) \in -\fR_+$;} \\
1 & \text{otherwise,}
\end{cases}
\]
and where the first product is ordered in any fixed arbitrary way,
then since $\lambda$ is antidominant the composition of the product morphism with the map $g \mapsto g \cdot L_{\lambda}$ induces an isomorphism
$J \simto \dot{w}^{-1} \Gr_y$;
in other words, if we set
\[
J' := \prod_{\beta \in -\fR_+} \prod_{i=n_\beta - \langle \lambda,\beta \rangle}^{-1} U_{\beta,i},
\]
then the composition of the product morphism with the map $g \mapsto g \cdot L_0$ induces an isomorphism
\[
J' \simto z^{-\lambda} \dot{w}^{-1} \Gr_y.
\]
This shows in particular that $z^{-\lambda} \dot{w}^{-1} \Gr_y \subset \mathrm{T}_0$.

Now, assume that $y \in W_\ext^\res$. If $\beta \in -\fRs$, then by Lemma~\ref{lem:res-elements} we have $\langle \lambda, \beta \rangle = n_\beta$. Hence in $J'$ there is no factor corresponding to the opposite of a simple root, which implies that $z^{-\lambda} \dot{w}^{-1} \Gr_y \subset (\chi_0^{\mathrm{T}})^{-1}(0)$. Since 
\[
(\dot{w} \mathrm{S}_{\nu}) \cap \Gr_y = \dot{w} z^\lambda \cdot \bigl( \mathrm{S}_{\nu-\lambda} \cap (z^{-\lambda} \dot{w}^{-1} \cdot \Gr_y) \bigr),
\]
Lemma~\ref{lem:geom-cas-sha} then implies our second claim.
\end{proof}

Consider the ``twisted product'' 
\[
\Gr \wttimes \Gr := \cL G \times^{\cL^+ G} \Gr
\]
and the (proper) morphism $m : \Gr \wttimes \Gr \to \Gr$ induced by multiplication in $\cL G$. 
Given locally closed subschemes $X,Y \subset \Gr$, we can consider the locally closed subscheme $X \wttimes Y \subset \Gr \wttimes \Gr$ defined as $X' \times^{\cL^+ G} Y$ where $X'$ is the preimage of $X$ in $\cL G$. In particular, for $y \in W_\ext^S$ and $\mu \in \bY_+$ we have the twisted product
$\Gr_y \wttimes \Gr^\mu \subset \Gr \wttimes \Gr$; we will denote by
\[
m_{y,\mu} : \Gr_y \wttimes \Gr^\mu \to \Gr
\]
the morphism induced by $m$.

\begin{lem}
\label{lem:dimension}
Let $y \in W_\ext^S$, $\mu \in \bY_+$ and $\eta \in \bY$. 
Then we have
\[
\dim(m_{y,\mu}^{-1}(\dot{y} L_\eta)) \leq \langle \rho, \mu+\eta \rangle.
\]
Moreover, this inequality is strict if $y \in W_\ext^\res$ and $\eta \neq w_\circ(\mu)$.
\end{lem}

\begin{proof}
Let us write $y=w t_\lambda$ with $w \in W$ and $\lambda \in \bY$; then $\lambda$ is antidominant (see Lemma~\ref{lem:res-elements} and~\eqref{eqn:WS-Wres}).
We note that we have a decomposition into locally closed pieces
\[
\Gr_y \wttimes \Gr^\mu = \bigsqcup_{\nu \in \bY}  ((\dot{w} \mathrm{S}_{\nu}) \cap \Gr_y) \wttimes \Gr^\mu,
\]
with only finitely many nonempty terms in the right-hand side; to compute the dimension in this statement it therefore suffices to consider the intersections
\[
m_{y,\mu}^{-1}(\dot{y} L_\eta) \cap \bigl( ((\dot{w} \mathrm{S}_{\nu}) \cap \Gr_y) \wttimes \Gr^\mu \bigr) = m_{y,\mu}^{-1}(\dot{w} L_{\lambda + \eta}) \cap \bigl( ((\dot{w} \mathrm{S}_{\nu}) \cap \Gr_y) \wttimes \Gr^\mu \bigr)
\]
for all $\nu \in \bY$.

Let us consider some $\nu \in \bY$ such that $(\dot{w} \mathrm{S}_{\nu}) \cap \Gr_y \neq \varnothing$, and choose $g \in \cL U^+$ such that $\dot{w} g \cdot L_\nu \in (\dot{w} \mathrm{S}_{\nu}) \cap \Gr_y$. Then if $m_{y,\mu}^{-1}(\dot{w} L_{\lambda + \eta}) \cap \bigl( ((\dot{w} \mathrm{S}_{\nu}) \cap \Gr_y) \wttimes \Gr^\mu \bigr)$ is nonempty, there exists $a \in \Gr^\mu$ such that $\dot{w} g z^\nu \cdot a = \dot{w} L_{\lambda + \eta}$; in particular we have $\Gr^\mu \cap \mathrm{S}_{\lambda+\eta-\nu} \neq \varnothing$, which implies that
\begin{equation}
\label{eqn:dim-inequality}
\langle \rho,\mu+\lambda+\eta-\nu \rangle \geq 0,
\end{equation}
this inequality being strict unless $\lambda+\eta-\nu=w_\circ(\mu)$, i.e.~$\nu=\lambda+\eta-w_\circ(\mu)$.

We claim that if $m_{y,\mu}^{-1}(\dot{w} L_{\lambda + \eta}) \cap \bigl( ((\dot{w} \mathrm{S}_{\nu}) \cap \Gr_y) \wttimes \Gr^\mu \bigr)$ is nonempty, then the natural morphism
\[
m_{y,\mu}^{-1}(\dot{w} L_{\lambda + \eta}) \cap \bigl( ((\dot{w} \mathrm{S}_{\nu}) \cap \Gr_y) \wttimes \Gr^\mu \bigr) \to (\dot{w} \mathrm{S}_{\nu}) \cap \Gr_y
\]
is a locally closed immersion. Indeed, if we denote by $X_{y,\mu}$ the image of $\overline{\Gr_y} \wttimes \overline{\Gr^\mu}$ under the proper morphism $m$ (a closed subscheme of $\Gr$), then the canonical morphism
\[
\overline{\Gr_y} \wttimes \overline{\Gr^\mu} \to \overline{\Gr_y} \times X_{y,\mu}
\]
is a closed immersion.
If we denote by $Y_{y,\mu}^\eta \subset \overline{\Gr_y} \wttimes \overline{\Gr^\mu}$ the inverse image of $\overline{\Gr_y} \times \{\dot{w} L_{\lambda + \eta} \}$ under this map, then the natural morphism
\[
Y_{y,\mu}^\eta \to \overline{\Gr_y}
\]
is a closed immersion, and hence so is its restriction
\[
Y_{y,\mu}^\eta \cap \bigl( ((\dot{w} \mathrm{S}_{\nu}) \cap \Gr_y) \wttimes \overline{\Gr^\mu} \bigr) \to (\dot{w} \mathrm{S}_{\nu}) \cap \Gr_y
\]
to the preimage of $(\dot{w} \mathrm{S}_{\nu}) \cap \Gr_y$. Our claim follows, since $m_{y,\mu}^{-1}(\dot{w} L_{\lambda + \eta}) \cap \bigl( ((\dot{w} \mathrm{S}_{\nu}) \cap \Gr_y) \wttimes \Gr^\mu \bigr)$ is the intersection of the domain of the latter morphism with the open subscheme $((\dot{w} \mathrm{S}_{\nu}) \cap \Gr_y) \wttimes \Gr^\mu$.

This claim implies that whenever $m_{y,\mu}^{-1}(\dot{w} L_{\lambda + \eta}) \cap \bigl( ((\dot{w} \mathrm{S}_{\nu}) \cap \Gr_y) \wttimes \Gr^\mu \bigr)$ is nonempty we have
\[
\dim \bigl( m_{y,\mu}^{-1}(\dot{y} L_\eta) \cap \bigl( ((\dot{w} \mathrm{S}_{\nu}) \cap \Gr_y) \wttimes \Gr^\mu \bigr) \bigr)
\leq \dim((\dot{w} \mathrm{S}_{\nu}) \cap \Gr_y).
\]
By Corollary~\ref{cor:dim-intersection} the right-hand side is at most $\langle \rho, \nu-\lambda \rangle$; combining this observation with~\eqref{eqn:dim-inequality} we deduce that
\[
\dim \bigl( m_{y,\mu}^{-1}(\dot{y} L_\eta) \cap \bigl( ((\dot{w} \mathrm{S}_{\nu}) \cap \Gr_y) \wttimes \Gr^\mu \bigr) \bigr)
\leq \langle \rho, \mu+\eta \rangle,
\]
which implies our first claim, and moreover that this inequality is strict unless $\nu=\lambda+\eta-w_\circ(\mu)$.

If we furthermore assume that $y \in W_\ext^\res$ and $\eta \neq w_\circ(\mu)$, then $\lambda+\eta-w_\circ(\mu) \neq \lambda$. The second claim in Corollary~\ref{cor:dim-intersection} shows that
\[
\dim \bigl( (\dot{w} \mathrm{S}_{\lambda+\eta-w_\circ(\mu)}) \cap \Gr_y \bigr) < \langle \rho, \eta-w_\circ(\mu) \rangle = \langle \rho, \mu+\eta \rangle,
\]
which shows our second claim.
\end{proof}

We finish this subsection with a reminder on some aspects of the geometric Satake equivalence (see~\S\ref{ss:Satake-category}) that will be used in our proofs below. Recall the $\cL^+G$-equivariant derived category $\Db_{\cL^+ G}(\Gr,\bk)$, its subcategory of perverse sheaves $\Perv_{\cL^+ G}(\Gr,\bk)$, and the (exact) convolution product $\star^{\cL^+G}$ introduced in~\S\ref{ss:Satake-category}.
Below we will use the fact that the monoidal category
\[
(\Perv_{\cL^+ G}(\Gr,\bk), \star^{\cL^+G})
\]
is rigid: every object $\cF$ has a left and right dual $\cF^\vee$. (This fact can either be checked directly or deduced from the geometric Satake equivalence.) We will not need an explicit description of this operation, but only that for $\mu \in \bY_+$ we have
\begin{equation}
\label{eqn:duals}
(\cI_!^\mu)^\vee \cong \cI_*^{-w_\circ(\mu)}, \quad (\cI_*^\mu)^\vee \cong \cI_!^{-w_\circ(\mu)}, \quad (\IC^\mu)^\vee \cong \IC^{-w_\circ(\mu)}.
\end{equation}

Our proof will also make use of the following result.

\begin{prop}
\label{prop:mathieu-thm}
For any $\lambda, \mu \in \bY_+$ the object $\cI_!^\lambda \star^{\cL^+G} \cI_!^\mu$ admits a filtration with subquotients of the form $\cI_!^\nu$ with $\nu \in \bY_+$. Dually, for any $\lambda, \mu \in \bY_+$ the object $\cI_*^\lambda \star^{\cL^+G} \cI_*^\mu$ admits a filtration with subquotients of the form $\cI_*^\nu$, with $\nu \in \bY_+$.
\end{prop}

This result is a geometric version of a theorem on tensor products of modules with good filtrations (for reductive algebraic groups over fields of positive characteristic) first due to Mathieu~\cite{mathieu} in full generality. It can be deduced from this result using the geometric Satake equivalence; a direct geometric proof can also be obtained from~\cite[Theorem~4.16]{bgmrr}, see~\cite{jmw2} for some details.

%-------------------------------------------------------------
\subsection{Proofs in case \texorpdfstring{$A=\varnothing$}{A=0}}
%-------------------------------------------------------------

We can now come to the proofs of the special case $A=\varnothing$ of Theorems~\ref{thm:geometric-Steinberg} and~\ref{thm:geometric-Steinberg-ff}. The following result is a consequence of Lemma~\ref{lem:dimension} that will be required below.

\begin{lem}
\label{lem:vanishing-Steinberg-2}
Let $y \in W_\ext^\res$, let $\mu \in \bY_+$, and let $\eta \in -\bY_+$. We have
\[
\Hom_{\Db_{I_\unip}(\Gr,\bk)}(\DGr_y \star^{\cL^+G} \cI_!^\mu, \NGr_{yt_\eta}[1])=0,
\]
and moreover
\[
\Hom_{\Db_{I_\unip}(\Gr,\bk)}(\DGr_y \star^{\cL^+G} \cI_!^\mu, \NGr_{yt_\eta})=0
\]
if $\eta \neq w_\circ(\mu)$. In case $\eta = w_\circ(\mu)$, we have
\[
\Hom_{\Db_{I_\unip}(\Gr,\bk)}(\DGr_y \star^{\cL^+G} \cI_!^\mu, \NGr_{yt_{w_\circ(\mu)}}) \neq 0.
\]
\end{lem}

\begin{proof}
Let $\cQ = ({}^{\mathrm{p}}\hspace{-1pt}\tau^{\le -1}(j^\mu_!\underline{\bk}_{\Gr^\mu}[\langle 2\rho,\mu \rangle]))[1]$, so that we have a distinguished triangle
\[
j^\mu_! \underline{\bk}_{\Gr^\mu}[\langle 2\rho,\mu \rangle] \to \cI^\mu_! \to \cQ \to.
\]
Note that $\cQ \in {}^{\mathrm{p}}\Db_{\cL^+G}(\Gr,\bk)^{\le -2}$.  By the t-exactness of $\star^{\cL^+G}$ and the fact that $\DGr_y$ and $\NGr_{yt_\eta}$ are perverse, we see that
\[
\Hom(\DGr_y \star^{\cL^+G} \cQ, \NGr_{yt_\eta}) = \Hom(\DGr_y \star^{\cL^+G} \cQ, \NGr_{yt_\eta}[1]) = 0.
\]
Thus, to prove the lemma it is enough to show that the space
\[
\Hom(\DGr_y \star^{\cL^+G} (j^\mu_! \underline{\bk}_{\Gr^\mu}[\langle 2\rho,\mu \rangle]), \NGr_{yt_\eta}[i])
\]
vanishes if $i = 1$ or if $i = 0$ and $\eta \ne w_\circ(\mu)$, and is nonzero
if $i=0$ and $\eta=w_\circ(\mu)$.

Using the notation introduced in~\S\ref{ss:proof-steinberg}, from the definition we see that
\[
\DGr_y \star^{\cL^+G} (j^\mu_! \underline{\bk}_{\Gr^\mu}[\langle 2\rho,\mu \rangle]) = (m_{y,\mu})_! \underline{\bk}[\ell(y)+\langle 2\rho, \mu \rangle];
\]
by the base change theorem we deduce that
\begin{multline*}
\Hom(\DGr_y \star^{\cL^+G} (j^\mu_! \underline{\bk}_{\Gr^\mu}[\langle 2\rho,\mu \rangle]), \NGr_{yt_\eta}[i]) \\
\cong \Hom \bigl( (m_{y,\mu}^{y t_\eta})_! \underline{\bk} [\ell(y)+\langle 2\rho, \mu \rangle],\underline{\bk}_{\Gr_{y t_\eta}} [\ell(yt_\eta)+i] \bigr),
\end{multline*}
where $m_{y,\mu}^{y t_\eta}$ is the restriction of $m_{y,\mu}$ to the preimage of $\Gr_{y t_\eta}$.
Now, $\Gr_{yt_\eta}$ is isomorphic to an affine space, and by equivariance the cohomology sheaves of $(m_{y,\mu}^{y t_\eta})_! \underline{\bk}$ are constant sheaves.  The $\Hom$-group above may therefore be computed after passing to stalks at $\dot{y} L_\eta \in \Gr_{yt_\eta}$. 
We deduce that
\[
\Hom(\DGr_y \star^{\cL^+G} (j^\mu_! \underline{\bk}_{\Gr^\mu}[\langle 2\rho,\mu \rangle]), \NGr_{yt_\eta}[i])
\cong \mathsf{H}_c^{\ell(y)+\langle 2\rho,\mu \rangle - \ell(yt_\eta) - i}(m_{y,\mu}^{-1}(\dot{y} L_\eta);\bk)^*.
\]
Here by Lemma~\ref{lem:length-res-dom} we have $\ell(yt_\eta)=\ell(y)-\langle 2\rho,\eta \rangle$, so that
\begin{equation}\label{eqn:fiber-cohom}
\Hom(\DGr_y \star^{\cL^+G}  (j^\mu_! \underline{\bk}_{\Gr^\mu}[\langle 2\rho,\mu \rangle]), \NGr_{yt_\eta}[i]) \cong \mathsf{H}_c^{\langle 2\rho,\mu+\eta \rangle-i}(m_{y,\mu}^{-1}(\dot{y} L_\eta);\bk)^*.
\end{equation}
By Lemma~\ref{lem:dimension}, if $\eta \neq w_\circ(\mu)$ we have $\dim(m_{y,\mu}^{-1}(\dot{y} L_\eta)) < \langle \rho, \mu+\eta \rangle$, so the right-hand side of~\eqref{eqn:fiber-cohom} vanishes for $i = 0$ and $i = 1$.  If $\eta = w_\circ(\mu)$, then we have
\[
\dim(m_{y,\mu}^{-1}(\dot{y} L_\eta)) \leq \langle \rho, \mu+\eta \rangle=0
\]
(again by Lemma~\ref{lem:dimension}) and $m_{y,\mu}^{-1}(\dot{y} L_\eta) \neq \varnothing$ (since $[\dot{y}: L_{\eta}] \in m_{y,\mu}^{-1}(\dot{y} L_\eta)$); the right-hand side of~\eqref{eqn:fiber-cohom} therefore still vanishes for $i = 1$, and is nonzero for $i = 0$.
\end{proof}

We are now ready to prove Theorem~\ref{thm:geometric-Steinberg-ff} in the special case $A=\varnothing$.

\begin{proof}[Proof of Theorem~\ref{thm:geometric-Steinberg-ff} when $A=\varnothing$]
The proof will consist of five steps.

\textit{Step 1. If $\cF,\cG \in \Perv_{\cL^+G}(\Gr,\bk)$ and if $\cF$ has a standard filtration and $\cG$ has a costandard filtration, then
\[
\Hom(\DGr_y \star^{\cL^+G} \cF, \NGr_y \star^{\cL^+G} \cG[1]) = 0.
\]}%
Of course we can assume that $\cG=\cI^\nu_*$ for some $\nu \in \bY_+$.
In the case where $\nu=0$, this claim follows from Lemma~\ref{lem:vanishing-Steinberg-2}.  The general case reduces to this case using the isomorphism
\[
\Hom(\DGr_y \star^{\cL^+G} \cF, \NGr_y \star^{\cL^+G} \cG[1])
\cong
\Hom(\DGr_y \star^{\cL^+G} \cF \star^{\cL^+G} \cG^\vee, \NGr_y [1]),
\]
since $\cF \star^{\cL^+G} \cG^\vee$ has a standard filtration by~\eqref{eqn:duals} and Theorem~\ref{prop:mathieu-thm}.

\textit{Step 2. Let $c: \DGr_y \to \NGr_y$ be the canonical map.  If $\cF \in \Perv_{\cL^+G}(\Gr,\bk)$ has a standard filtration and $\cG \in \Perv_{\cL^+G}(\Gr,\bk)$ has a costandard filtration, then the map
\[
\Hom(\cF, \cG) \to \Hom(\DGr_y \star^{\cL^+G} \cF, \NGr_y \star^{\cL^+G} \cG)
\qquad\text{given by}\qquad
\phi \mapsto c \star^{\cL^+G} \phi
\]
is an isomorphism.}
First we assume that $\cG = \cI_*^0$; in this case we will prove the claim by induction on the length of a standard filtration of $\cF$.  If $\cF = \cI^\mu_!$ for some $\mu \in \bY_+$, then both sides vanish if $\mu \ne 0$ by Lemma~\ref{lem:vanishing-Steinberg-2}, and the map is clearly an isomorphism if $\mu=0$.
Now suppose $\cF$ has a standard filtration of length${}> 1$, and choose some short exact sequence $0 \to \cF' \to \cF \to \cI^\mu_! \to 0$ where $\cF'$ has a standard filtration and $\mu \in \bY_+$.  We have a commutative diagram
\[
\hbox{\small$
\begin{tikzcd}[column sep=tiny]
0 \ar[r] &
\Hom(\cI^\mu_!, \cI_*^0) \ar[r] \ar[d, "\wr"] &
\Hom(\cF, \cI_*^0) \ar[r] \ar[d] &
\Hom(\cF', \cI_*^0) \ar[r] \ar[d, "\wr"] &
\Hom(\cI^\mu_!,\cI_*^0[1]) \ar[d, "\wr"] \\
0 \ar[r] &
\Hom(\DGr_y {\star} \cI^\mu_!, \NGr_y) \ar[r] &
\Hom(\DGr_y {\star} \cF, \NGr_y) \ar[r] &
\Hom(\DGr_y {\star} \cF', \NGr_y) \ar[r] &
\Hom(\DGr_y {\star} \cI^\mu_!,\NGr_y[1]),
\end{tikzcd}$}
\]
where we write $\star$ for $\star^{\cL^+G}$ and all vertical arrows are as in the claim.
The first and third vertical arrows are isomorphisms by induction, and the fourth vertical arrow is an isomorphism because both terms vanish (by Step 1).  By the five lemma, the second vertical arrow is also an isomorphism, finishing the proof in this case. Once this case is established, we deduce the general case by adjunction, as in Step 1.

\textit{Step 3. If $\cF,\cG \in \Perv_{\cL^+G}(\Gr,\bk)$ and if $\cF$ has a standard filtration and $\cG$ has a costandard filtration, then the map
\[
\Hom(\cF, \cG) \to \Hom(\LGr_y \star^{\cL^+G} \cF, \LGr_y \star^{\cL^+G} \cG)
\]
is an isomorphism.}  This follows from the observation that the map from Step 2 is the composition of the map above with the natural map
\[
\Hom(\LGr_y \star^{\cL^+G} \cF, \LGr_y \star^{\cL^+G} \cG) \to
\Hom(\DGr_y \star^{\cL^+G} \cF, \NGr_y \star^{\cL^+G} \cG),
\]
which is injective by t-exactness of $\star^{\cL^+ G}$.

\textit{Step 4.  If $\cF \in \Perv_{\cL^+G}(\Gr,\bk)$ has a standard filtration, and $\cG \in \Perv_{\cL^+G}(\Gr,\bk)$ is arbitrary, then the map
\[
\Hom(\cF, \cG) \to \Hom(\LGr_y \star^{\cL^+G} \cF, \LGr_y \star^{\cL^+G} \cG)
\]
is an isomorphism.}  Let $Z \subset \Gr$ be the support of $\cG$; this is a closed union of finitely many $\cL^+G$-orbits. By results of~\cite{mv} (see~\cite[\S 1.12.1]{bar}), the category $\Perv_{\cL^+G}(Z,\bk)$ admits a projective generator which admits a standard filtration. By duality, it therefore also admits an injective generator which admits a costandard filtration; in particular there exists
a copresentation
\[
0 \to \cG \to \cI \to \cI'
\]
where $\cI,\cI' \in \Perv_{\cL^+G}(\Gr,\bk)$ have costandard filtrations.  We then have a commutative diagram
\[
\begin{tikzcd}[column sep=small, row sep=small]
0 \ar[r] &
\Hom(\cF, \cG) \ar[r] \ar[d] &
\Hom(\cF, \cI) \ar[r] \ar[d, "\wr"] &
\Hom(\cF, \cI') \ar[d, "\wr"] \\
0 \ar[r] &
\Hom(\LGr_y \star \cF, \LGr_y \star \cG) \ar[r] &
\Hom(\LGr_y \star \cF, \LGr_y \star \cI) \ar[r] &
\Hom(\LGr_y \star \cF, \LGr_y \star \cI'),
\end{tikzcd}
\]
where we again write $\star$ for $\star^{\cL^+G}$.
The last two vertical maps are isomorphisms by Step~3, so by the five lemma the first is as well.

\textit{Step 5. Proof of full faithfulness of $\LGr_y \star^{\cL^+G} ({-})$ in general.}  This is very similar to Step~4, using a presentation of $\cF$ by perverse sheaves with standard filtrations.
\end{proof}

Using the special case of Theorem~\ref{thm:geometric-Steinberg-ff} proved above, we can now deduce the corresponding special case of Theorem~\ref{thm:geometric-Steinberg}.

\begin{proof}[Proof of Theorem~\ref{thm:geometric-Steinberg} when $A=\varnothing$]
First we claim that $\LGr_y \star^{\cL^+G} \IC^\mu$ is simple. Indeed, otherwise there exists a surjective and noninjective morphism $\LGr_y \star^{\cL^+G} \IC^\mu \twoheadrightarrow \cF$ for some simple object $\cF$ in $\Perv_{I_\unip}(\Gr,\bk)$. Now convolution commutes with Verdier duality, so $\LGr_y \star^{\cL^+G} \IC^\mu$ is self-dual.  Since the simple object $\cF$ is also self-dual, we can apply Verdier duality to obtain an injective and nonsurjective morphism $\cF \hookrightarrow \LGr_y \star^{\cL^+G} \IC^\mu$. Composing these two maps, we obtain a nonzero endomorphism of $\LGr_y \star^{\cL^+G} \IC^\mu$ which is not a multiple of the identity, proving that
\[
\dim \Hom_{\Db_{I_\unip}(\Gr,\bk)}(\LGr_y \star^{\cL^+G} \IC^\mu,\LGr_y \star^{\cL^+G} \IC^\mu) \geq 2,
\]
and therefore contradicting (the known special case of) Theorem~\ref{thm:geometric-Steinberg-ff}.

On the other hand, we claim that the perverse sheaf $\LGr_y \star^{\cL^+G} \IC^\mu$ admits $\LGr_{yt_{w_\circ(\mu)}}$ as a composition factor.  (This claim will complete the proof.) Indeed, we have a surjection
\begin{equation}
\label{eqn:surjection-D-L-conv}
\DGr_y \star^{\cL^+G} \IC^\mu \twoheadrightarrow \LGr_y \star^{\cL^+G} \IC^\mu.
\end{equation}
Here 
since $\NGr_{yt_{w_\circ(\mu)}}$ has $\LGr_{yt_{w_\circ(\mu)}}$ as socle,
the fact that $\Hom(\DGr_y \star^{\cL^+G} \IC^\mu,\NGr_{yt_{w_\circ(\mu)}}) \neq 0$ (see Lemma~\ref{lem:vanishing-Steinberg-2})
implies that $\DGr_y \star^{\cL^+G} \IC^\mu$ admits $\LGr_{yt_{w_\circ(\mu)}}$ as a composition factor. For dimension reasons the support of the kernel of~\eqref{eqn:surjection-D-L-conv} does not intersect $\Gr_{yt_{w_\circ(\mu)}}$, so this kernel does not admit $\LGr_{yt_{w_\circ(\mu)}}$ as a composition factor, which implies the desired claim.
\end{proof}

\begin{rmk}
\label{rmk:vanishing-Steinberg-2}
The reasoning at the end of the preceding proof can be used to make Lemma~\ref{lem:vanishing-Steinberg-2} a little bit more precise: in the notation of this statement we have
\[
\dim \left( \Hom_{\Db_{I_\unip}(\Gr,\bk)}(\DGr_y \star^{\cL^+G} \cI_!^\mu, \NGr_{yt_{w_\circ(\mu)}}) \right) = 1.
\]
Indeed, this follows from the observation that the support of the kernel of~\eqref{eqn:surjection-D-L-conv}
does not meet $\Gr_{yt_{w_\circ(\mu)}}$.
\end{rmk}

%------------------------------------------------------------------------
\subsection{Proofs for general \texorpdfstring{$A$}{A}}
\label{ss:Steinberg-Whittaker}
%------------------------------------------------------------------------

We now prove Theorems~\ref{thm:geometric-Steinberg} and~\ref{thm:geometric-Steinberg-ff} for a general finitary subset $A \subset S_\aff$.

\begin{proof}[Proof of Theorem~\ref{thm:geometric-Steinberg}]
Recall the functor $\Av^A_\psi$ studied in~\S\ref{ss:Av-Whit}. It is clear that this functor commutes with convolution on the right by $\cL^+G$-equivariant objects.  Applying $\Av^A_\psi$ to the isomorphism $\LGr_y \star^{\cL^+G} \IC^\mu \cong \LGr_{yt_{w_\circ(\mu)}}$ from the case $A=\varnothing$ and using Lemma~\ref{lem:AvWhit}\eqref{it:AvWhit-4}, we deduce the desired isomorphism.
\end{proof}

\begin{proof}[Proof of Theorem~\ref{thm:geometric-Steinberg-ff}]
Since the case $A=\varnothing$ is now known, it is enough to show that for $\cF, \cG \in \Perv_{\cL^+G}(\Gr,\bk)$ the map
\begin{equation}\label{eqn:Whit-Stein-ff}
\Hom(\LGr_y \star^{\cL^+G} \cF, \LGr_y \star^{\cL^+G} \cG) \to 
\Hom(\LGr^A_y \star^{\cL^+G} \cF, \LGr^A_y \star^{\cL^+G} \cG)
\end{equation}
induced by $\Av^A_\psi$ is an isomorphism.  
Since the functor $(-) \star^{\cL^+ G} \cF$ is left adjoint to $(-) \star^{\cL^+ G} \cF^\vee$,
we may (and will) assume without loss of generality that $\cF = \IC^0$.  

First, we claim that the map
\begin{equation}
\label{eqn:Whit-Stein-dl}
\Hom(\DGr_y, \LGr_y \star^{\cL^+G} \cG) \to 
\Hom(\DGr^A_y, \LGr^A_y \star^{\cL^+G} \cG)
\end{equation}
induced by $\Av^A_\psi$
is an isomorphism. Indeed, by adjunction we have
\begin{multline*}
\Hom(\DGr^A_y, \LGr^A_y \star^{\cL^+G} \cG) 
\cong \Hom(\Av^A_\psi(\DGr_y), \Av^A_\psi(\LGr_y \star^{\cL^+G} \cG) )
\\
\cong \Hom(\Av^A_!(\Av^A_\psi(\DGr_y)), \LGr_y \star^{\cL^+G} \cG).
\end{multline*}
Thus,~\eqref{eqn:Whit-Stein-dl} can be identified with the map
\begin{equation}\label{eqn:Whit-Stein-dl2}
\Hom(\DGr_y, \LGr_y \star^{\cL^+G} \cG) \to 
\Hom(\Av^A_!(\Av^A_\psi(\DGr_y)), \LGr_y \star^{\cL^+G} \cG)
\end{equation}
induced by the adjunction morphism $f : \Av^A_!(\Av^A_\psi(\DGr_y)) \to \DGr_y$.  By Lemma~\ref{lem:Av-Iw}\eqref{it:Av-Iw-2} we have $\Av^A_!(\Av^A_\psi(\DGr_y)) \cong \Av^A_!(\DFl^A_e) \star^I \DGr_y$, and our map is induced by a surjection $\Av^A_!(\DFl^A_e) \to \LFl_e$. It follows from~\cite[Lemma~10.1]{br} that the kernel of this surjection admits a filtration with subquotients of the form $\DFl_v$ for $v \in W_A \smallsetminus \{e\}$, each appearing once. Since $y$ is minimal in $W_A y$, we have $\ell(vy)=\ell(v)+\ell(y)$ for any such $v$, from which one deduces that $\DFl_v \star^I \DGr_y \cong \DGr_{vy}$ by standard arguments; we deduce that $f$ is surjective, and that its kernel $\mathcal{K}$ has a standard filtration with subquotients of the form $\DGr_{vy}$ with $v \in W_A \smallsetminus \{ e\}$.  Thus~\eqref{eqn:Whit-Stein-dl2} is injective, and to prove that it is surjective it suffices to show that $\Hom(\mathcal{K}, \LGr_y \star^{\cL^+G} \cG) = 0$, which will follow if we prove that
\[
\Hom(\DGr_{vy}, \LGr_y \star^{\cL^+G} \cG) = 0
\]
when $v \in W_A \smallsetminus \{e\}$.  This holds because the unique simple quotient of $\DGr_{vy}$, namely $\LGr_{vy}$, does not occur as a composition factor of $\LGr_y \star^{\cL^+G} \cG$, since these composition factors are of the form $\LGr_{yt_{w_\circ}(\nu)}$ for some $\nu \in \bY_+$ by Theorem~\ref{thm:geometric-Steinberg}, and thus in particular have their label in ${}^A W^S_\ext$. (Note that $vy$ does not belong to ${}^A W^S_\ext$ since it is not minimal in the coset $W_A vy=W_A y$.)

Next, we consider the commutative diagram
\[
\begin{tikzcd}[row sep=small]
\Hom(\LGr_y, \LGr_y \star^{\cL^+G} \cG)  \ar[r] \ar[d, "\text{\eqref{eqn:Whit-Stein-ff}}"'] &
  \Hom(\DGr_y, \LGr_y \star^{\cL^+G} \cG)  \ar[d, "\text{\eqref{eqn:Whit-Stein-dl}}"] \\
\Hom(\LGr^A_y, \LGr^A_y \star^{\cL^+G} \cG) \ar[r] &
  \Hom(\DGr^A_y, \LGr^A_y \star^{\cL^+G} \cG).
\end{tikzcd}
\]
Here the right-hand vertical arrow is an isomorphism as proved above, and both horizontal maps are injective, because they are induced by the surjective morphisms $\DGr_y \to \LGr_y$ and $\DGr^A_y \to \LGr^A_y$ respectively. The upper arrow is in fact even an isomorphism, since the kernel of the surjection $\DGr_y \to \LGr_y$ has its composition factors of the form $\LGr_z$ with $\ell(z) < \ell(y)$, while all composition factors of $\LGr_y \star^{\cL^+G} \cG$ are of the form $\LGr_{yt_\nu}$ with $\nu \in -\bY_+$ by Theorem~\ref{thm:geometric-Steinberg}, and hence have their label of length at least $\ell(y)$ by Lemma~\ref{lem:length-res-dom}. 
We deduce that all four maps above are isomorphisms, which finishes the proof.
\end{proof}

%----------------------------------------------------------------
\subsection{A conjecture on the image of \texorpdfstring{$\Phi_{y,A}$}{PhiyA}}
\label{ss:conj}
%----------------------------------------------------------------

A special case of the functor in Theorem~\ref{thm:geometric-Steinberg-ff} has already appeared in the literature: it is the case when $A=S$ (so that $W_A=W$) and $y=t_\varsigma w_\circ$. For these choices,  it is shown in~\cite{bgmrr} that this functor is in fact an equivalence of categories. For general $A$ and $y$, this functor cannot be an equivalence, simply because not every simple object of $\Perv_{(I_\unip^A, \cX_A)}(\Gr,\bk)$ belongs to its essential image. But one might still expect that in some cases it satisfies a property stronger than full faithfulness.
Namely, denote by $\sC_{y,A} \subset \Perv_{(I_\unip^A,\cX_A)}(\Gr,\bk)$ the Serre subcategory generated by the simple objects of the form $\LGr^A_{yt_\lambda}$ with $\lambda \in -\bY_+$. Then by exactness and Theorem~\ref{thm:geometric-Steinberg}, the functor $\Phi_{y,A}$ factors through a (fully faithful and exact) functor
$\Perv_{\cL^+G}(\Gr,\bk) \to \sC_{y,A}$,
which will still be denoted $\Phi_{y,A}$.

\begin{conj}
\label{conj}
Assume that $y \in {}^A W_\ext^\res$ is minimal in ${}^A W_\ext^S$ for the Bruhat order. Then the functor
\[
\Phi_{y,A} : \Perv_{\cL^+G}(\Gr,\bk) \to \sC_{y,A}
\]
is an equivalence of categories.
\end{conj}

We see Conjecture~\ref{conj} as giving ``partially Whittaker models" for the Satake category, in the spirit of~\cite{bgmrr}. As an evidence for this conjecture, let us note that it holds at least in the following cases:
\begin{enumerate}
\item
\label{it:conj-empty}
when $A=\varnothing$ and $y=e$;
\item
\label{it:conj-S}
when $A=S$ and $y=t_\varsigma w_\circ$;
\item
\label{it:char-0}
when $\bk$ has characteristic $0$.
\end{enumerate}
In fact, in case~\eqref{it:conj-empty} this claim is equivalent to the standard result---due to Mirkovi{\'c}--Vilonen~\cite{mv}---that the forgetful functor from $\Perv_{\cL^+G}(\Gr,\bk)$ to the category of perverse sheaves constructible with respect to the stratification by $\cL^+G$-orbits is an equivalence, see~\cite[Proposition~2.1]{mv}. In case~\eqref{it:conj-S}, the functor $\Phi_{t_\varsigma w_\circ,S}$ is the main object of study of~\cite{bgmrr}; in this setting we have $\sC_{t_\varsigma w_\circ,S}=\Perv_{(I_\unip^S,\cX_S)}(\Gr,\bk)$ by Lemma~\ref{lem:min-LR}, $t_\varsigma w_\circ$ is minimal for the Bruhat order because it has minimal length in ${}^S W_\ext^S$ (by the same statement and Lemma~\ref{lem:length-res-dom}), and the main result of~\cite{bgmrr} states that this functor is an equivalence of categories. (Note that revisiting the arguments in~\cite[\S 4.3]{bgmrr} involving parity complexes, one can prove directly that $\Phi_{t_\varsigma w_\circ,S}$ is essentially surjective once we know that it is fully faithful.) Finally, in case~\eqref{it:char-0}, parity considerations imply that the category $\sC_{y,A}$ is semisimple, which of course implies the statement.

%----------------------------------------------------------------
\subsection{An example}
%----------------------------------------------------------------

%We note finally that o
One can check that, given $y \in {}^A W_\ext^\res$ minimal in ${}^A W_\ext^S$ for the Bruhat order, Conjecture~\ref{conj} is equivalent to the statement that
\begin{equation}
\label{eqn:assumption-Whit-Serre}
\Ext^1_{\Perv_{(I_\unip^A,\cX_A)}(\Gr,\bk)}(\DGr^A_y \star^{\cL^+G} \cI_!^\mu, \NGr^A_y) = 0 \quad \text{for any $\mu \in \bY_+$.}
\end{equation}
Let us denote by
\[
m^A_{y,\mu} : \Gr^A_y \wttimes \Gr^\mu \to \Gr
\]
the morphism induced by $m$. As in the proof of Lemma~\ref{lem:vanishing-Steinberg-2}, we have an embedding
\begin{equation}
\label{eqn:conj-embedding}
\Ext^1_{\Perv_{(I_\unip^A,\cX_A)}(\Gr,\bk)}(\DGr^A_y \star^{\cL^+G} \cI_!^\mu, \NGr^A_y) \hookrightarrow \mathsf{H}_c^{\langle 2\rho,\mu \rangle-1} \bigl( (m^A_{y,\mu})^{-1}(\dot{y} L_0);\mathcal{F} \bigr)^*
\end{equation}
where $\mathcal{F}$ is the restriction to $(m^A_{y,\mu})^{-1}(\dot{y} L_0)$ of the pullback of the rank-$1$ $(I_\unip^A,\cX_A)$-equivariant local system on $\Gr^A_y$.
Since
$\Gr^A_y = \dot{w}_A \cdot \Gr_{w_A y}$,
left multiplication by $\dot{w}_A^{-1}$ induces an isomorphism
\[
(m^A_{y,\mu})^{-1}(\dot{y} L_0) \simto m_{w_Ay, \mu}^{-1}(\dot{w}_A^{-1} \dot{y} L_0).
\]
The right-hand side is of the form studied in Lemma~\ref{lem:dimension}; if $w_A y$ is restricted then
\[
\dim((m^A_{y,\mu})^{-1}(\dot{y} L_0)) < \langle \rho,\mu \rangle
\]
provided $\mu \neq 0$,
which allows one to deduce~\eqref{eqn:assumption-Whit-Serre} in this case. Unfortunately, this condition is not always satisfied, and it can happen that $\dim((m^A_{y,\mu})^{-1}(\dot{y} L_0)) = \langle \rho,\mu \rangle$. 

Indeed, consider the case $G=\mathrm{GL}_2(\F)$, with the standard choice of maximal torus and (negative) Borel subgroup, and $A=S$. Here we have a canonical identification $\bY=\Z^2$, such that $\bY_+ = \{(a,b) \in \Z^2 \mid a \geq b\}$, and we can take $\varsigma=(1,0)$. If $s$ is the unique element in $S$, then $y=t_\varsigma s$ is restricted and minimal in ${}^S W_\ext^S$ (in fact $\ell(y)=0$), but $sy=t_{(0,1)}$ is not restricted. One can check that in this case we have
\[
(m^S_{t_{(1,0)} s,(1,-1)})^{-1}(L_{(1,0)}) = \left\{ \bigl[ (\begin{smallmatrix} z & x \\ 0 & 1 \end{smallmatrix}) : (\begin{smallmatrix} 1 & -xz^{-1} \\ 0 & 1 \end{smallmatrix}) G(\F[ \hspace{-1pt} [z] \hspace{-1pt} ]) \bigr] : x \in \F^\times \right\};
\]
in particular, this scheme has dimension $1=\langle \rho, (1,-1) \rangle$. Here $\mathcal{F}$ is the restriction of the Artin--Schreier local system, so that the right-hand side in~\eqref{eqn:conj-embedding} has dimension $1$. 

This example illustrates why our proof of Theorem~\ref{thm:geometric-Steinberg-ff} has to be different in case $A \neq \varnothing$. (Note that in any case Conjecture~\ref{conj} is known in the special case considered here, as explained in~\S\ref{ss:conj}.)

%%%%%%%%%%%%%%%%%%
%%%%%%%%%%%%%%%%%%

\end{document}